\documentclass{amsart}

\usepackage{tikz-cd}
\usepackage{ amssymb }
\usepackage{hyperref}
\hypersetup{
    colorlinks=true,
    filecolor=magenta,      
    urlcolor=cyan,
}
\newtheorem{theorem}{Theorem}[section]
\newtheorem{lemma}[theorem]{Lemma}
\newtheorem{proposition}[theorem]{Proposition}

\theoremstyle{definition}
\newtheorem{definition}[theorem]{Definition}

\theoremstyle{remark}
\newtheorem{remark}[theorem]{Remark}

\numberwithin{equation}{section}

 \usepackage[OT2,T1]{fontenc}
  \DeclareSymbolFont{cyrletters}{OT2}{wncyr}{m}{n}
  \DeclareMathSymbol{\sha}{\mathalpha}{cyrletters}{"58}

\newcommand{\spec}{\text{Spec} \ }
\newcommand{\inj}{\hookrightarrow}
\newcommand{\Of}{\mathcal{O}}
\newcommand{\N}{\mathbb{N}}
\newcommand{\R}{\mathbb{R}}
\newcommand{\Rp}{\R_{\geq 0}}

\newcommand{\Hh}{\mathbb{H}}
\newcommand{\Ld}{\mathbb{L}}
\newcommand{\n}{\mathfrak{n}}
\newcommand{\m}{\mathfrak{m}}

\newcommand{\Y}{\mathcal{Y}}
\newcommand{\nn}{\mathcal{N}}

\newcommand{\xxx}{\mathbb{X}}

\newcommand{\ii}{\mathcal{I}}
\newcommand{\dR}{\mathrm{dR}}
\newcommand{\Q}{\mathbb{Q}}
\newcommand{\upm}{U^{\pm 1}}

\newcommand{\Z}{\mathbb{Z}}
\newcommand{\C}{\mathbb{C}}
\newcommand{\st}{\text{st}}

\newcommand{\Star}{\mathrm{Star}}

\newcommand{\ruo}{\text{if}}
\newcommand{\gal}{\mathrm{Gal}}

\newcommand{\an}{\mathrm{an}}
\newcommand{\pt}{\mathrm{pt}}
\newcommand{\expo}{\mathrm{exp}}

\begin{document}
\title{Ordinarity of Local Galois Representation Arising from Dwork Motives}

\author{Lie Qian}
\address{Department of Mathematics, Stanford University}
\curraddr{Department of Mathematics, Stanford University}
\email{lqian@stanford.edu}
\thanks{}



\date{January 25th, 2021}

\dedicatory{}

\maketitle

\keywords{}

\section{Introduction}

Let $F$ be a characteristic $0$ local field containing $\zeta_N$ whose residue characteristic equals to $p$ and does not divide $N$. We first introduce the Dwork motives. 

Let $T_{0}=\mathbb{P}^{1}-(\{\infty\}\cup\mu_{N})/\mathbb{Z}[1/N]$ with coordinate $t$ and $Z\subset \mathbb{P}^{N-1}\times T_{0}$ be a projective family defined by the following equation:
$$
X_{1}^{N}+X_{2}^{N}+ \cdots +X_{N}^{N}=NtX_{1}X_{2}\cdots X_{N}
$$

The map $\pi:Z\rightarrow T_{0}$ is a smooth of relative dimension $N-2$. We will write $Z_{s}$ for the fiber of this family at a point $s$. Let $H=\mu_{N}^{N}/\mu_{N}$ and
$$
H_{0}:=\{(\xi_{1},\ldots, \xi_{N})\in\mu_{N}^{N}:\xi_{1}\cdots\xi_{N}=1\}/\mu_{N}\subset H
$$
Over $\mathbb{Z}[1/N, \zeta_{N}]$ there is an $H$ action on $Z$ by:
$$
(\xi_{1},\ldots, \xi_{N})(X_{1},\ldots, X_{N}, t)=(\xi_{1}X_{1},\ldots, \xi_{N}X_{N}, (\xi_{1}\cdots\xi_{N})^{-1}t)
$$
Thus $H_{0}$ acts on every fibre $Z_{s}$, and $H$ acts on $Z_{0}$.

Fix $\chi$ a character $H_{0}\rightarrow\mu_{N}$ of the form:
$$
\chi\ ((\xi_{1},\ldots, \xi_{N}))=\prod_{i=1}^{N}\xi_{i}^{a_{i}}
$$

where $(a_1, \ldots, a_N)$ are $N$ constants such that $\sum_{i=1}^{N}a_{i}\equiv 0$ mod $N$. Therefore the character is well-defined.

We define the Dwork motive to be given by $Z$ and the $\chi$-eigenpart of the $H_0$ group action. In concrete terms, its $p$-adic realization is defined below.

For any prime $\lambda$ of $\mathbb{Z}[1/2N, \zeta_{N}]$ of residue characteristic $p$, we define the lisse sheaf $V_{\lambda}/(T_{0}\times{\rm Spec} \mathbb{Z}[1/2Np, \zeta_{N}])_{et}$ by:
$$
V_{\lambda}=(R^{N-2}\pi_{\ast}\mathbb{Z}[\zeta_{N}]_{\lambda})^{\chi,H_{0}}
$$
here the $\chi, H_0$ in the supscript means the $\chi$-eigenpart of the $H_0$ action.

We let $V_{\lambda, t}$ denote the fibre of the sheaf $V_\lambda$ over $t$  for a $t\in T_0(F')$, where $F'$ is a finite extension of $F$. In other words, viewed as a $G_{F'}$ representation, $V_{\lambda, t}$ is just $H^{N-2}(Z_{\overline{t}}, \Z[\zeta_N]_{\lambda})^{\chi, H_0}$, here $\overline{t}$ is the corresponding geometric point of $t$.

Fix the embedding $\tau:\mathbb{Q}(\zeta_{N})\inj\mathbb{C}$  such that $\tau(\zeta_{N})=e^{2\pi i/N}$. Let $\tilde{\pi}: Y(\mathbb{C})\rightarrow T_{0}(\mathbb{C})$ denote the base change of $\pi$ along $\tau$ and $V_{B}$ be the locally constant sheaf over $T_{0}(\mathbb{C})$:
$$
V_{B}=(R^{N-2}\tilde{\pi}_{\ast}\mathbb{Z}[\zeta_{N}])^{\chi,H_{0}}
$$

By standard comparison results(see e.g. \cite{BLGHT}), $V_\lambda$ and $V_B$ are locally constant and locally free of the same rank over $\Z[\zeta_N]_\lambda$ and $\Z[\zeta_N]$. This rank can be computed by looking at the fibre over $0$. Denote this rank by $n$.

Fix a nonzero base point $s\in T_{0}(\mathbb{C})$. Now we have the monodromy representation:
$$
\rho_{s}:\pi_{1}(T_{0}(\mathbb{C}), s)\rightarrow GL(V_{B,s})
$$
Let $\gamma_{\infty}$ be the loop around $ \infty$ as an element of $\pi_{1}(T_{0}(\mathbb{C}), s)$.

We can now state the main theorem of this paper.

\begin{theorem}
\label{mt}
Under the assumption that for the motives defined by any fixed $\sigma\chi$, where $\sigma\in\gal(\Q(\mu_N)/\Q)$ is arbitrary, $\rho_{s}(\gamma_{\infty})$ {\it has minimal polynomial} $(X-1)^{n}$, i.e. it is maximally unipotent, we have that the $G_{F_0}$ representation $V_{\lambda, t^{-1}}$ is regular ordinary for any finite extension $F_0$  of $F$ and $t\in F_0$ with $v(t)>0$,  where $v$ is the valuation of $F_0$. 
\end{theorem}

\begin{remark}
Note that by the argument in Page 12-13 of the author's forthcoming work \textit{Potential Automorphy for $GL_n$}, the assumption that $\rho_{s}(\gamma_{\infty})$ has minimal polynomial $(X-1)^{n}$ is satisfied when $0\in \{a_1,\ldots, a_N\}$.
\end{remark}

The main application of our theorems would be in the proof of a potential automorphy theorem claimed in a forthcoming paper by the author. There we choose certain $(a_1, \ldots, a_N)$ so that the input condition of main theorem is satisfied. In that paper we need certain local Galois representation to be regular and ordinary in order to apply the automorphy lifting theorem from \cite{tap}.

We shall briefly discuss the idea of the proof of the theorem. Since we may replace $F$ by $F_0$ at the very beginning, we assume without loss of generality that $t$ is an $F$ point $T_0$ from now on.

The theorem will be implied by the following two theorems this paper is going to establish, each possibly interesting in its own right. 

The first theorem claims the existence of a  semistable model for $Z_t$ with $t$ of negative valuation over a finite extension $F'/F$. The semistable model comes from a series of blowup of the naive integral model given by the equation $t^{-1}(X_1^N+X_2^N+\cdots+X_N^N)=NX_1X_2\cdots X_N$ in  $\mathbb{P}^{N-1}\times \Of_{F'}$ for some extension $F'/F$. We first work with the fundamental case where $t$ is of (normalized) valuation $-1$. Then we use the technique of toroidal embedding (in a mixed charateristic setting) established by Mumford to go from this case to the general case. This is the main result of section 2. 

The second theroem gives a way to compute the log cristalline cohomology 
\[
H^{\ast}_{crys}((\overline{Y_1}, \overline{N_1})/(W(k), \N, 1\mapsto 0))
\]
(See section 3 for a brief recall of the notation concerning log geometry) with its $N$ operator in terms of log de-Rham cohomology under the setting that $(\overline{Y_1}, \overline{N_1})$ is log smooth over $(k, \N, 1\mapsto 0)$ and can be lifted into a family with log structure $(Y, N)/(W(k)[T], \N, 1\mapsto T)$. Indeed, we use the log de Rham cohomology of the characteristic $0$ lift given by the fibre over $0$ of $(Y, N)$ and the operator $N$ is given as the connecting homomorphism of certain long exact sequence that is similar to the one used to define the Gauss-Manin connection. This is the first main result of section 3. 

From now on, when there are no ambiguity of the residue field $k$, we will use $W_n$ to denote $W_n(k)$ and $W$ to denote $W(k)$.

We will now explain these 2 results in more detail. 

\subsection{Semistable Model}

Let us recall what we mean by semistable throughout this paper, following \cite{HK}. 

\begin{definition}
\label{smst}
We say a scheme $X$ over a discrete valuation ring $A$ is with semistable reduction if etale locally on $X$, there is a smooth morphism $X\rightarrow A[T_1, \ldots, T_r]/(T_1\cdots T_r-\pi)$ for some $r\geq 0$, where $\pi$ is a uniformizer of the ring $A$.
\end{definition}

\begin{theorem}
\label{ssbp}
For any $t\in F$ such that $v(t)=d>0$,  there exists a totally ramified extension $F'/F$ generated by $\pi^{1/e}$, a choice of $e$-th root of $\pi$ which is a uniformizer of $\Of_F$ (thus $F'$ is purely ramified of degree $e$ over $F$ with uniformizer $\pi^{1/e}$), and a semistable model $\Y$ over $\Of_{F'}$ of $Z_{t^{-1}}$  with compatible $H_0$ action. 

\end{theorem}

Here is a supplement of the above theorem. It gives some idea on how the semistable model $\Y$ is constructed, and will be used as a link to reduce our computation to the situation over the unramified base ring $W$, where Hyodo-Kato's log geometry technique could be applied.

Let $\spec W[S, \upm]'$ denote the open subscheme of $\spec W[S, \upm]$ defined by $(S^{de}U)^N\neq 1$ and let $W[S, \upm]'$ denote the ring of regular function of this scheme.

\begin{proposition}
\label{ss}
In the setting of the above theorem, we can actually find a variety $Z$ with $H_0$ action over $\spec W[S, U^{\pm 1}]'$  that is a blowup of the variety $US^{de}(X_1^N+X_2^N+\cdots+X_N^N)=NX_1X_2\cdots X_N$ over $\spec W[S, U^{\pm 1}]'$  (the latter is the base change of the previous $UT(X_1^N+X_2^N+\cdots+X_N^N)=NX_1X_2\cdots X_N$ along $W[T, U^{\pm 1}]\rightarrow W[S, U^{\pm 1}]'$, $T\mapsto S^{de}$) that is an isomorphism outside the closed subscheme defined by $S$, such that locally $Z$ admits an etale map to 
\begin{equation}
W[U^{\pm 1}, S, Z_1, \ldots, Z_r, Z_{r+1}^{\pm 1}, \ldots, Z_{n}^{\pm 1}]/(US-Z_1\cdots Z_r)
\end{equation} 
over $ W[S, \upm]'$

Choosing a uniformizer $\pi$ of $\Of_F$ and write $t=u\pi^d$ with $u\in \Of_F^\times$, the $\Y$ in the above theorem is obtained from base change of $Z$ along $W[S, U^{\pm 1}]'\rightarrow \Of_{F'}$, $S\mapsto \pi^{1/e}, U\mapsto u$.

\end{proposition}

The link to a model over $W$ will be stated in Remark \ref{link}.

\subsection{Log Geometry}

A general log scheme is usually denoted as $(Z, M)$ for a scheme $Z$ and a a sheaf of monoid $M$. However, we would sometime use a third argument after $M$ to show what the structure map $M\rightarrow \Of_Z$ is. Again, we will use $\spec W[T]'$ to denote a fixed choice of affine open subscheme of $\spec W[T]$ such that the closed subscheme defined by $T=0$ is contained in $\spec W[T]'$. Let $\spec W_n[T]''$ be the mod $p^n$ reduction of $\spec W[T]'$.

Let $(Y, N)$ be a log smooth scheme over $(W[T]', \N, 1\mapsto T)$.  Denote by $(Y_n, N_n)$ (resp. $(\overline{Y_n}, \overline{N_n})$) the base change of $(Y, N)$ along $(W_n[T]', \N, 1\mapsto T)\rightarrow (W[T]', \N, 1\mapsto T)$(resp. $(W_n, \N, 1\mapsto 0)\rightarrow (W[T]', \N, 1\mapsto T)$). Let $(\overline{Y}, \overline{N})$ be the base change of $(Y, N)$ along $(W, \N, 1\mapsto 0)\rightarrow (W[T]', \N, 1\mapsto T)$.  Note that the closed immersions in the corresponding fibre diagram are all exact and the projection maps to the base are all log smooth.

Hyodo-Kato define the $i$-th log cristalline cohomology of $(\overline{Y_1}, \overline{N_1})$
as the limit
\begin{equation}
\left(\varprojlim_{n}H^{i}\left(\left((\overline{Y_1}, \overline{N_1})/(W_n,\N, 1\mapsto 0)\right)_{crys},\Of_{\overline{Y_1}/W_n}\right)\right) [\frac{1}{p}]
\end{equation}
and equip it with a nilpotent operator $N$.

\begin{theorem}
\label{criscp}
The above limit is isomorphic to $\Hh^i(\dR_{\overline{Y}/(W, \N, 1\mapsto 0)})[\frac{1}{p}]$  while the operator $N$ is given as the degree $i$ boundary homomorphism of the long exact sequence given by the following exact triangle
    
\begin{tikzcd}
\dR_{\overline{Y}/(W, \N, 1\mapsto 0)}[-1]\arrow[r, "\cdot dlog 1"]
&\dR_{\overline{Y}/(W, (0))}\arrow[r]
&\dR_{\overline{Y}/(W, \N, 1\mapsto 0)}\arrow[r]
&{}
\end{tikzcd}

which is defined at each degree $i$ by taking the $i$-th wedge power  of  locally split exact sequence of locally free sheaves of modules
\[
\begin{tikzcd}
0\arrow[r]&
\Of_{\overline{Y}}\arrow[r, "\cdot d\log 1"]&
\omega^1_{\overline{Y}/(W, (0))}\arrow[r]&
\omega^1_{\overline{Y}/(W, \N, 1\mapsto 0)}\arrow[r]&
0
\end{tikzcd}
\]
given in  Theorem 3.2.3 of \cite{ogus} since we note that $(\overline{Y},\overline{N})$ is log smooth over $(W, \N, 1\mapsto 0)$ and we identify $\omega^1_{(W, \N, 1\mapsto 0)/(W, (0))}\cong W$ by $d\log 1\mapsto 1$.
\end{theorem}

\begin{remark}
In the setting of the above theorem, note also that the resulting exact sequence 
\[
\begin{tikzcd}
0\arrow[r]&
\omega^{i-1}_{\overline{Y}/(W, \N, 1\mapsto 0)}\arrow[r, "\cdot d\log 1"]&
\omega^i_{\overline{Y}/(W, (0))}\arrow[r]&
\omega^i_{\overline{Y}/(W, \N, 1\mapsto 0)}\arrow[r]&
0
\end{tikzcd}
\]
at degree $i$ is a locally split exact sequence of locally free sheaves of modules.

\end{remark}

Finally, let $(\overline{Y}_\C, \overline{N}_\C)$ be a proper log scheme smooth over $(\C, \N, 1\mapsto 0)$ and $(\overline{Y^\an}, \overline{N^\an})$ be the analytic log scheme associated to it which is also smooth over the analytic point $(\pt, \N, 1\mapsto 0)$. We would like to reduce the computation to the analytic setting by the following theorem. It will be proved by GAGA.

\begin{theorem}
\label{injty}
There exists an $N$ equivariant isomorphism
\[
\Hh^{i}(\dR_{\overline{Y}_\C/(\C, \N, 1\mapsto 0)})\cong \Hh^{i}(\dR_{\overline{Y^\an}/(\pt, \N, 1\mapsto 0)})
\]
where the $N$ on each hypercohomology is defined by degree $i$ boundary morphism of the exact triangles obtained similar to the one in Theorem \ref{criscp}

\begin{tikzcd}
\dR_{\overline{Y}_\C/(\C, \N, 1\mapsto 0)}[-1]\arrow[r, "\cdot dlog 1"]
&\dR_{\overline{Y}_\C/(\C, (0))}\arrow[r]
&\dR_{\overline{Y}_\C/(\C, \N, 1\mapsto 0)}\arrow[r]
&{}
\end{tikzcd}

\begin{tikzcd}
\dR_{\overline{Y^\an}/(\pt, \N, 1\mapsto 0)}[-1]\arrow[r, "\cdot dlog 1"]
&\dR_{\overline{Y^\an}/(\pt,(0))}\arrow[r]
&\dR_{\overline{Y^\an}/(\pt, \N, 1\mapsto 0)}\arrow[r]
&{}
\end{tikzcd} 

\end{theorem}

\subsection*{Acknowledgements.}
I would like to first thank Richard Taylor for encouraging me to think about the subject of this paper. I also want to thank him for all the helpful comments on the draft of this paper. I am grateful to Richard Taylor, Brian Conrad, Weibo Fu, Ravi Vakil, Bogdan Zavyalov for many interesting conversation during the preparation of this text.

\section{Existence of Semistable Blowup}
In this section we always assume char$(k)\nmid N$ and $N>1$. Again, we will use $W_n$ to denote $W_n(k)$ and $W$ to denote $W(k)$

We first prove a lemma that basically settles the case where $t$ is a uniformizer in \ref{ssbp} via base change along $W[T, \upm]'\rightarrow \Of_F$, $T\mapsto t, U\mapsto u$, where $\spec W[T, \upm]'$ is the open subscheme of $\spec W[T, \upm]$ defined by $(TU)^{N}\neq 1$. The idea of this blowup process comes from a construction of Nick Shepherd-Barron.

\begin{lemma}
\label{blup}
There exists a blowup $\xxx$ of the variety $UT(X_1^N+X_2^N+\cdots+X_N^N)=NX_1X_2\cdots X_N$ over $\spec W[T, \upm]'$ with an action of $H_0$   such that the blowdown map is $H_0$-equivariant and  etale locally  $\xxx$ admits an etale map to 
\[
W[T, \upm]'[Z_1, \ldots, Z_r, Z_{r+1}^{\pm 1}, \ldots, Z_{n}^{\pm 1}]/(TU-Z_1\cdots Z_r)
\]
over $\spec W[T, \upm]'$ 
\end{lemma}

\begin{proof}

The sequence of blowup is as the following. Initially we have divisors $D_i$ defined by $X_i$ and $D$ defined by $X_1^N+X_2^N+\cdots+X_N^N$. Denote the original variety as $Y_N$, we blowup along $D\cap D_N$ to get a variety $Y_{N-1}$, as well as the divisors $D_i^{(N-1)}$ and $D^{(N-1)}$ as the strict transform of $D_i$ and $D$ respectively. In general, by induction we will have the variety $Y_j$ at the $(N-j)$-th stage, with divisors $D_i^{(j)}$ and $D^{(j)}$ as the strict transform of $D_i^{(j+1)}$ and $D^{(j+1)}$, $\forall i\in \{1,\ldots, n\}$, then we blowup along $D^{(j)}\cap D_j^{(j)}$, we get the variety $Y_{j-1}$ with the divisors $D_i^{(j-1)}$ and $D^{(j-1)}$. The final step is that we let $\xxx:=Y_0$ be the blowup of $Y_1$ along $D^{(1)}\cap D_1^{(1)}$. Note that for each blowup $Y_{j-1}\rightarrow Y_j$, we can associate an $H_0$ action on $Y_{j-1}$ such that the morphism is $H_0$-equivariant because we may see by induction that the locus $D^{(j)}\cap D_j^{(j)}$ we are blowing up along is $H_0$-stable.

We check the desired local property of $Y_0$ by hand. Fix an $i\in \{1,\ldots, N\}$ from now on and we will only work with the preimage of the affine chart $ X_i\neq 0 $ and we will still use $Y_k$ to denote the primage of this chart in the original $Y_k$. In each step of the blowups, we might write the blowup as the union of two affine open subschemes given by the complement of the strict transform of the two divisors whose intersection we are blowing up along. And note that the affine open given by the complement of $D^{(j)}$ will be isomorphic to its preimage under all further blowup, since it has empty intersection with the locus we are blowing up along. 

More precisely, let $l$ be the function defined on $\{1,\ldots, \hat{i}, \ldots, N\}$ by the rule $l(x)=x+1$ if $x\neq i-1$ while $l(i-1)=i+1$. For $0<k\leq N$ with $k\neq i$, let $U_k$ denote the variety
\[
\spec (W[T,\upm]'[x_1,\ldots,\hat{x_i},\ldots, x_N, b_k', b_{l(k)}]/(Nb_k'\prod_{1\leq j\leq k-1, j\neq i}x_j-UT,
\]
\[
b_k'b_{l(k)}-x_k, b_{l(k)}\prod_{l(k)\leq j\leq N, j\neq i}x_j-\sum_{1\leq j\leq N, j\neq i}x_j^N-1)
\]
here $x_j$ are the affine coordinates $\frac{X_j}{X_i}$, with the understanding that when $k=N$, the product $\prod_{l(k)\leq j\leq N, j\neq i}x_j=1$, and thus $b_{l(N)}=\sum_{1\leq j\leq N, j\neq i}x_j^N+1$.

And for $0<k<N$ with $k\neq i$, let $V_k$ denote the variety \[
\spec (W[T,\upm]'[x_1,\ldots,\hat{x_i},\ldots, x_N, b_{l(k)}]/(N\prod_{1\leq j\leq k, j\neq i}x_j-UTb_{l(k)},
\]
\[
b_{l(k)}\prod_{l(k)\leq j\leq N, j\neq i}x_j-\sum_{1\leq j\leq N, j\neq i}x_j^N-1)
\]

It can be seen inductively for $k$, the following properties hold.
\begin{itemize}
    \item $Y_k=V_k\cup \cup_{l(k)\leq j\leq N, j\neq i}U_j$ as affine opens for $0\leq k\leq N, k\neq i$, $Y_i=Y_{i+1}$.
    
    \item for any $m\leq k$ $D_m^{(k)}$ is the divisor given by $x_m$ in the affine open $V_k$ and $U_j$ for all $l(k)\leq j\leq N, j\neq i$. 
    
    \item for $m\geq l(k)$, $D_m^{(k)}$ is the divisor given by $b_m'$ in $U_m$ and $x_m$ in $U_j$ for all $j>m $ and it has empty intersection with the rest of the affine opens.
    
    \item $D^{(k)}$ is given by the divisor $b_{l(k)}$ in the affine open $V_k$ only.

\end{itemize}

Thus the blowup is an isomorphism over each $U_k$ in the affine charts of $Y_j$ and maps $V_{j-1}\cup U_j$ to $V_{j}$ in the next step. From this, it suffice to verify all $U_k$ satisfy the local property. Now $U_k$ can be written as
\[
\spec W[ \upm,x_1, \ldots,\hat{x_i},\hat{x_k},\ldots, x_N, b_k', b_{l(k)}]/(
b_{l(k)}\prod_{l(k)\leq j\leq N, j\neq i}x_j-
\]
\[
\sum_{1\leq j\leq N, j\neq i,k}x_j^N-(b_k'b_{l(k)})^N-1)
\]

We first prove $U_k$ is regular by applying Jacobian criterion to this variety. We fix a closed point $q$ of this variety giving rise to a $\overline{k}$-point of the form $$x_1\mapsto a_1,\ldots, x_N \mapsto a_N, b_k'\mapsto u, b_{l(k)}\mapsto v, U\mapsto w,$$
for some constant $a_j, u,v,w\in \overline{k}$. Let $\m$ be the maximal ideal in \linebreak $W[ \upm,x_1, \ldots,\hat{x_i},\hat{x_k},\ldots, x_N, b_k', b_{l(k)}]$. It suffice to check that the defining equation is not $0$ in $\m/\m^2\otimes_{k(\m)}\overline{k}$, which has a basis $dp, dx_1,\ldots, dx_N, db_k', db_{l(k)}, dU$.

Let $f$ be the defining equation
\[
b_{l(k)}\prod_{l(k)\leq j\leq N, j\neq i}x_j-\sum_{1\leq j\leq N, j\neq i,k}x_j^N-(b_k'b_{l(k)})^N-1
\]
of $U_k$. Then $df=\sum_{j\neq i,k}f_jdx_j+f_udb_k'+f_vdb_{l(k)}$ where 
\[
\begin{array}{ll}
    f_j=-Na_j^{N-1}& \ruo \ j<k\\
    f_j=v\prod_{l(k)\leq s\leq N, s\neq i,j}a_s-Na_j^{N-1}& \ruo \ j\geq l(k)\\
    f_u=-Nu^{N-1}v^N&\\
    f_v=-Nu^Nv^{N-1}+\prod_{l(k)\leq s\leq N, s\neq i}a_s&
\end{array}
\]

If all of the coefficients are $0$, we see that $a_j=0$, $\forall j<k$ from $f_j=0$ since $l\nmid N$. From $f_u=0$ we see $u=0$ or $v=0$, either implies by $f_v=0$ that $a_s=0$ for some $l(k)\leq s\leq N, s\neq i$. This implies $a_j=0$, $\forall j\neq i,k,s$ by the $f_j=0$, which in turn implies $a_s=0$ by the $f_s=0$. But substituting the known zeroes into $f$ we see that this can never happen. Thus $U_k$ is regular.

To verify the local property, fix a $U_k$ to work in.

(1) If $f_j\neq 0$ for some $j\geq l(k)$, then we claim that  the map \linebreak $U_k\rightarrow W[\upm, x_1,\ldots,\hat{x_i}, \hat{x_j}, \hat{x_k},\ldots, x_N, b_k', b_{l(k)}]$ given by corresponding coordinate is etale near the point $q$. Let $\n$ denote the maximal ideal corresponding to the image of $q$ under this map. Then the claim follows because 
\[
\begin{split}
    \m/\m^2\otimes_{k(\m)}\overline{k}&\cong\overline{k}dp\oplus \overline{k}dx_1\oplus\cdots\oplus\overline{k}dx_N\oplus \overline{k}db_k'\oplus\overline{k}db_{l(k)}\oplus\overline{k}dU/\\
    &\left(\sum_{j\neq i,k}f_jdx_j+f_udb_k'+f_vdb_{l(k)}\right)\\
    &\cong\overline{k}dp\oplus \overline{k}dx_1\oplus\cdots\oplus \widehat{\overline{k}dx_j}\oplus\overline{k}dx_N\oplus \overline{k}db_k'\oplus\overline{k}db_{l(k)}\oplus\overline{k}dU\\
    &\cong\n/\n^2\otimes_{k(\n)}\overline{k}
\end{split}
\]
here the middle isomorphism follows because $f_j\neq 0$,
and we define the structure map $\spec\left(W[\upm, x_1,\ldots,\hat{x_i}, \hat{x_j}, \hat{x_k},\ldots, x_N, b_k', b_{l(k)}]\right)\rightarrow W[T, \upm]'$ by $T\mapsto NU^{-1}b_k'\prod_{1\leq m\leq k-1, m\neq i}x_m$, so that the morphism is a morphism of $W[T, \upm]'$ scheme.

(2) If $f_v\neq 0$, then same argument as that in (1) gives that the map $U_k\rightarrow W[\upm, x_1,\ldots,\hat{x_i},  \hat{x_k},\ldots, x_N, b_k']$ is etale near at the point $q$. And  again define the structure map by $T\mapsto NU^{-1}b_k'\prod_{1\leq m\leq k-1, m\neq i}x_m$.

Note now that if we try to use similar argument to "kill" the variable $x_j$ for some $j<k$, then it is hard to define a structure map on the target scheme because $x_j$ appears in the expression of $T$ in the original $U_k$. Hence we use the following trick.

(3) If $a_j^N\neq (uv)^N$ for some $j<k$, then we consider the map \linebreak $U_k\rightarrow W[\upm, x_1,\ldots,\hat{x_i}, \hat{x_j}, \hat{x_k},\ldots, x_N, \hat{b_k'}, b_{l(k)}, c_j]$ where every variable maps to the corresponding one in the structure ring of $U_k$ except $c_j\mapsto b_k'x_j$, thus $dc_j\mapsto udx_j+a_jdb_k'$ under the pullback of the map mentioned above. The condition   $a_j^N\neq (uv)^N$ gives that
\[
\begin{pmatrix}

  u   &  a_j \\
  f_j   & f_u

\end{pmatrix}
\]
is nondegenerate, and hence 
\[
\begin{split}
    \m/\m^2\otimes_{k(\m)}\overline{k}&\cong\overline{k}dp\oplus \overline{k}dx_1\oplus\cdots\oplus\overline{k}dx_N\oplus \overline{k}db_k'\oplus\overline{k}db_{l(k)}\oplus\overline{k}dU/\\
    &\left(\sum_{j\neq i,k}f_jdx_j+f_udb_k'+f_vdb_{l(k)}\right)\\
    &\cong\overline{k}dp\oplus \overline{k}dx_1\oplus\cdots\oplus \widehat{\overline{k}dx_j}\oplus\overline{k}dx_N\oplus \widehat{\overline{k}db_k'}\oplus\overline{k}db_{l(k)}\oplus\overline{k}dU\oplus \overline{k}(udx_j+a_jdb_k')\\
    &\cong\n/\n^2\otimes_{k(\n)}\overline{k}
\end{split}
\]
where $\n$ denote the maximal ideal corresoonding to the image of $q$ under the map. And we may take the structure map $W[\upm, x_1,\ldots,\hat{x_i}, \hat{x_j}, \hat{x_k},\ldots, x_N, \hat{b_k'}, b_{l(k)}, c_j]\rightarrow W[T, \upm]'$ to be $T\mapsto NU^{-1}c_j\prod_{1\leq m\leq k-1, m\neq i, j}x_m$. It could be checked that the map is a map of $W[T, \upm]'$ scheme.

We conclude that this exhausts all possibilities if we impose the condition that $q$ lies in the closed subscheme defined by $T$: If all of the three conditions above do not hold, then  $a_j^N=(uv)^N$ for all $j<k$, $f_j=0$ for all $j\geq l(k)$ and $f_v=0$. 

We claim for all $j>l(k)$, $a_j\neq 0$: If $a_j=0$ for some $j\geq l(k)$, then checking the condition $f_{j'}=0$ for all $j'\geq l(k)$, $j'\neq j$, gives $a_j=0$ for all $j\geq l(k)$. Putting that into $f_v=0$ gives $u=0$ or $v=0$. Putting that once more into the condition  $a_j^N=(uv)^N$ gives $a_j=0$ for all $j<k$. Thus, we see that all $f_j$, $j\neq i,k$, $f_u$ and $f_v$ are $0$. This have been shown to be impossible by the argument for regularity, hence we prove the claim that for all $j>l(k)$, $a_j\neq 0$.

Multiplying each $f_j=0$ by $a_j$ gives $a_{l(k)}^N=\cdots =a_N^N=\frac{1}{N}v\prod_{l(k)\leq s\leq N, s\neq i}a_s$ and thus $v\neq 0$. Hence it also follows from $f_v=0$ that $u\neq 0$. The condition $a_j^N=(uv)^N$ then gives $a_j\neq 0$ for all $j<k$. Therefore, $T=NU^{-1}b_k'\prod_{1\leq m\leq k-1, m\neq i}a_m$ is not $0$ at $q$ because  $u$, $a_m$ for $1\leq m\leq k-1, m\neq i$ are all nonzeron as we have seen.

Similar consideration shows that for characteristic $0$ points, the only peculiartiy can happen in the open locus $T\neq 0$.

Now over the locus $T\neq 0$, the blowdown map is an isomorphism since the locus we are blowing up against is always a codimension $1$ irreducible subvariety. Hence it suffices to show that the original projective variety $UT(X_1^N+X_2^N+\cdots+X_N^N)=NX_1X_2\cdots X_N$ admits such an etale map locally. Again we only work with the affine chart $X_i\neq 0$ and writing the intersection with this chart as the affine variety $U_0$
\[
\spec W[T, \upm]'[x_1,\ldots, \hat{x_i},\ldots, x_N]/(UT(1+x_1^N+\cdots+\hat{x_i^N}+\cdots+x_N^N)-N\prod_{j\neq i}x_j)
\]
Take a geometric point $q$ with coordinates $x_j$ sent to $b_j$ in some algebraically closed field. Denote the defining equation by $g$ and its derivatives with respect to each variable $x_j$ evaluated at $q$ as $g_j$. Similar to the argument in (1), the map $U_0\rightarrow \spec W[T, \upm]'[x_1,\ldots, \hat{x_i},\hat{x_j, }\ldots, x_N]$ is etale at $q$ if $g_j\neq 0$. Hence such a map (clearly over $W[T, \upm]'$) exist as long as one of the $g_j\neq 0$. 

If all $g_j=0$, then $UTb_j^{N-1}=\prod_{k\neq j,i}b_k$ for any $j\neq i$. In other words, $b_1^N=\cdots=b_N^N=U^{-1}T^{-1}\prod_{j\neq i}b_j$. Denote this constant by $C$. Also the defining equation gives us $1+b_1^N+\cdots+\hat{b_i^N}+\cdots+b_N^N=NU^{-1}T^{-1}\prod_{j\neq i}b_j$. Hence $1+(N-1)C=NC$ and it follows that $C=1$. Putting it back to the equations give that $UT=\zeta_N$ for an $N$-th root of unity. We have excluded this from the base, thus we conclude that the local expression over $W[T,\upm]'$ exists near all points on $\xxx$.

\end{proof}

\begin{remark}
\label{h0act}
Under the notation of the proof above, we can concretely describe the action of $H_0$ on each $U_k$ as the following: $(\xi_{1},\ldots, \xi_{N})$ acts by $x_j\mapsto \frac{\xi_j}{\xi_i}x_j$ for $1\leq j\leq N, j\neq i,k$, $T\mapsto T$, $b_k'\mapsto \left(\prod_{1\leq j\leq k-1, j\neq i}\frac{\xi_i}{\xi_j}\right)b_k'$ and $b_{l(k)}\mapsto \left(\prod_{l(k)\leq j\leq N, j\neq i}\frac{\xi_i}{\xi_j}\right)b_{l(k)}$.

Thus we may also make all the local etale maps given in the proof to be $H_0$-equivariant with the $H_0$ action on the target be the corresponding multiplication on each coordinates.
\end{remark}

We have seen there is a variety denoted by $\xxx$, that is a blowup with $H_0$ action of the variety $UT(X_1^N+X_2^N+\cdots+X_N^N)=NX_1X_2\cdots X_N$  over $\spec W[T, \upm]'$ and is an isomorphism outside the closed subscheme defined by $T$. $\xxx$ locally admits an etale map to the form given as in \ref{blup}. Thus $\xxx$ further satisfies the property that the base change of it along $W[T, \upm]'\rightarrow \Of_F$, $T\mapsto \pi, U\mapsto u$ 
is a semistable(in the sense of Definition \ref{smst}) model  of its generic fibre, with $H_0$ action. The generic fibre is the underlying motive of the $l$-adic representation $V_{\lambda, (u\pi)^{-1}}$. That follows because the generic fibre factor through the open locus of $\xxx$ where $T\neq 0$. We can use this semistable model and take $Z=\xxx$ to prove Theorem \ref{ssbp} and Proposition \ref{ss} in the case $d=1$.

Base changing the blowdown map we get from Lemma \ref{blup}   along $W[T, \upm]'\rightarrow W[R, \upm]'$, $T\mapsto R^d$, where $\spec W[R, \upm]'$ is the open subscheme of $\spec W[R, \upm]$ given by $(R^dU)^N\neq 1$, we immediately get a blowup $\xxx_d$ of the projective variety $UR^d(X_1^N+X_2^N+\cdots+X_N^N)=NX_1X_2\cdots X_N$ over $\spec W[R, \upm]'$ with an action of $H_0$,  such that the blowdown map is $H_0$-equivariant and $\xxx_d$ locally admits an etale map to some
\[
W[R, \upm]'[Z_1, \ldots, Z_r, Z_{r+1}^{\pm 1}, \ldots, Z_{n}^{\pm 1}]/(UR^d-Z_1\cdots Z_r)
\]

Note that to proceed now, we cannot simply base change $X$ along \linebreak $W[R, \upm]'\rightarrow \Of_F$, $R\mapsto \pi, U\mapsto u$, since then the model has etale locally the form 
\[
\Of_F[Z_1, \ldots, Z_r, Z_{r+1}^{\pm 1}, \ldots, Z_{n}^{\pm 1}]/(u\pi^d-Z_1\cdots Z_r)
\]
which is not semistable.

We state the theorem we would like to prove in this section below.

\begin{theorem}
For any $t\in F$ such that $v(t)=d>0$,  there exists a (e and a extension) totally ramified extension $F'/F$ generated by $\pi^{1/e}$, a choice of $e$-th root of $\pi$ which is a uniformizer of $\Of_F$ (thus $F'$ is purely ramified of degree $e$ over $F$ with uniformizer $\pi^{1/e}$), and a semistable model $\Y$ over $\Of_{F'}$ of $Z_{t^{-1}}$ (defined in the first page)  with compatible $H_0$ action.  

\end{theorem}

The idea of the proof goes as the following. Since the $\xxx_d$ above is a blowup of the projective variety $UR^d(X_1^N+X_2^N+\cdots+X_N^N)=NX_1X_2\cdots X_N$ over $\spec W[R, \upm]'$, its base change along $W[R, \upm]'\rightarrow F$, $R\mapsto \pi, U\mapsto u$ is isomorphic to the underlying motive of $V_{\lambda, t^{-1}}$ where $t=u\pi^d$. We will use the theory of toroidal embedding to construct a blowup $Z$ of $\xxx_{de}$ for some $e$ along some locus contained in  the closed subscheme defined by $S$, where $\xxx_{de}$ is the base change of $\xxx_d$ along $W[R, \upm]'\rightarrow W[S, \upm]'$, $T\rightarrow S^e$ (again $\spec W[S, \upm]'$ is defined by the condition $(S^{de}U)^N\neq 1$ in $\spec W[S, \upm]$ as mentioned in the introduction). We want our $Z$ to admit an etale morphism to 
\begin{equation}
W[S, \upm, Z_1, \ldots, Z_r, Z_{r+1}^{\pm 1}, \ldots, Z_{n}^{\pm 1}]/(US-Z_1\cdots Z_r)
\end{equation}
Zariski locally so that we see from construction that its base change along \linebreak $W[S, \upm]\rightarrow \Of_{F'}$, $S\mapsto \pi^{1/e}, U\mapsto u$ gives our desired $\Y$.

We will use the theory of toroidal embeddings in the mixed characteristic case from \cite{Mum4} to reduce construction to a combinatorial problem of subdividing conical complexes, which is also resolved in \cite{Mum3}. So we will sketch the main idea and notation from \cite{Mum1} \cite{Mum2} \cite{Mum3} \cite{Mum4} first.

Let $\eta$ be the generic point of $\spec W$ and we will work with the base $\spec W$. A torus embedding is an irreducible normal variety $X$ of finite type over $\spec W$ with inclusion $T_\eta\inj X$, here $T_\eta$ is the generic fiber, with $T$ acting on $X$ extending the action of $T$ on $T_\eta$. 

Let $M(T), N(T)$ be the character and cocharacter group of $T$. Let $\widetilde{M}(T)$ ($\widetilde{N}(T)$) be $\Z\times M(T)$ ($\Z\times N(T)$) and $N_\R (T)=N(T)\otimes \R$ ($\widetilde{N}_\R (T)=\widetilde{N}(T)\otimes \R$). Like the usual case over a field, affine torus embeddings all come from cones $\sigma$ in $\R_{\geq 0}\times N_\R(T)$ not containing any linear subspace in $N_\R(T)$. The association is as the following: 
\label{deftor}
\[
X_\sigma=\spec R[\cdots, \pi^k\cdot \mathfrak{X}^\alpha, \cdots]_{(k, \alpha)\in \widetilde{M}(T), \ \   \langle(k, \alpha), v\rangle\geq 0, \forall v\in \sigma}
\]
with the orbits of $T_\eta$ in $X_\eta$ corresponding to the faces of $\sigma$ in $0\times N_\R(T)$ and the orbit of $T_0$ in $X_0$ corresponding to the faces of $\sigma$ in $\R_{>0}\times N_\R(T)$. We will mostly deal with the simple case where $\sigma$ has the form $\R_{\geq 0}\times \sigma'$ for some cone $\sigma'\subset N_R(T)$. In that case the orbits of $T_\eta$ in $X_\eta$ correspond naturally to the orbits of $T_0$ in $X_0$ via specialization.

The above is somehow the local picture. We will call an irreducible normal scheme $X$ of finite type over $R$ with an open $U\subset X_\eta$ a toroidal embedding if Zariski locally near a point $x$, there exists an etale morphism $\rho$ from $(U, X)$ to some torus embedding $(T_\eta, Z_{\rho(x)})$ as pairs of schemes. This gives a global stratification of $X-U$ as the following. The existence of local models gives that $X-U=\cup_{i\in I} E_i$ with $E_i$ normal irreducible  subscheme(so no self intersection for these $E_i$) of codimension $1$. The strata are defined as the connected components of $\cap_{i\in J}E_i-\cup_{i\notin J}E_i$ for various subsets $J\subset I$, and for $J=\varnothing$, we take the strata to be $U$. Locally under the etale maps to the torus embedding models $Z_{\rho(x)}$, these strata admit etale maps into the strata of the local models  $Z_{\rho(x)}$, hence the original strata are regular because those of $Z_{\rho(x)}$ are, due to the explicit description. The point of this definition is that we may assign a 
canonical topological space $\Delta$ that is the union of the cones obtained locally from the models. The cones are glued together via inclusion contravariantly with respect to  the specialization of strata. The canonical topological space $\Delta$ reflects lots of properties of $X$. The precise definition of the conical complex $\Delta$ is in page 196 of \cite{Mum4}. It has the structure of a topological space and each cone in it admits an injection into some $\R^m$ with integral structure. We just stress that the main idea is that the character group $M(T)$ of the local model can be canonically defined by our toroidal embedding $(U, X)$ as the group of Cartier divisors on $\Star Y$ supported in $\Star Y-U$ for a strata $Y$ (Star means the union of all the strata that specialize to it) and the cone can be canonically described as the dual cone of the effective Cartier divisors.

Here is the crucial property we use to reduce our construction to the problem of subdividing a complex: Any subdivision $\Delta'$ of $\Delta$ determines a morphism $(U, X')\rightarrow (U, X)$ such that $(U, X')$ is a toroidal embedding with the associated complex $\Delta'$ and that the associated map of complexes $\Delta'\rightarrow\Delta$ is the natural inclusion map. Indeed, the map is a blowdown for an ideal sheaf on $X$ supported in $X-U$ completely determined by the subdivision $\Delta'$.

We also have interpretation in terms of the complex $\Delta$ of the properties that the toroidal embedding variety $X$ being regular and certain Cartier divisor being a sum of irreducible Weil divisors without repetition. In the case over a field, for a Cartier divisor $S$, by restricting to a Cartier divisor in $\Star Y$ we can view it as an element in the character group $M$ associated to this stratum and hence a function on the cone $\sigma $ associated to this stratum. Compatibility gives us that $S$ induce  a  globally defined function on $X$, still denoted by $S$.  Assume the hyperplane $S=1$ in $\Delta$ defined by this function  meets every nonzero faces in $\Delta$, then $X$ is regular and $S$ vanishes to order $1$ on all irreducible components of $X-U$ if and only if the intersection of $S=1$ with any face $\tau_\alpha$ of $\Delta$ has vertices with integral coordinates and the volume of the above intersection equals to $1/d_\alpha!$ for $d_\alpha$ the dimension of $\tau_\alpha$. We will use similar argument to Mumford's in the proof of this fact to prove a similar result in the mixed characteristic case. It essentially reduces to the same proof by adjoining the extra $\R$ factor.

\begin{proof}
To apply the theory of toroidal embeddings to get a blowup of $\xxx_d$ or $\xxx_{de}$, we first need to check $\xxx_d$ and $\xxx_{de}$ have  the structure of toroidal embeddings. Our setting is as the following: let $U_d\subset \xxx_{d, \eta}$ be the open subvariety  defined by $R\neq 0$. We know $\xxx_d$  locally admits an etale map to 
\begin{equation}
\label{etx}
W[R, \upm]'[Z_1, \ldots, Z_r, Z_{r+1}^{\pm 1}, \ldots, Z_{n}^{\pm 1}]/(UR^d-Z_1\cdots Z_r)
\end{equation}
which is an affine torus embedding of the generic torus $T_\eta$ of the split torus $T$ of the form
\[
W[R^{\pm 1}, \upm, Z_1^{\pm 1}, \ldots, Z_r^{\pm 1}, Z_{r+1}^{\pm 1}, \ldots, Z_{n}^{\pm 1}]/(UR^d-Z_1\cdots Z_r)
\]
with $T$ action given by multiplying corresponding coordinates. And clearly $U_d$ is the preimage of the generic torus via $R\neq 0$. Let $e_i$ denote the character given by the regular function $Z_i$ in $X^\ast(T)$, $e_u$ denote the character given by the regular function $U$ and similarly $f_i$ and $f_u$ in $X_\ast(T)$, then the lattice $\widetilde{M}(T)$ is identified with $\Z e_1+\ldots +\Z e_n+\Z e_u+ \Z e_0+ \Z\frac{1}{d}(-e_u+\sum_{j=1}^{r}e_j)$, where $e_0$ corresponds to the extra factor of $\Z$ as in page 191 of \cite{Mum4}. $\widetilde{N}_\R(T)$ is spanned by $f_1,\ldots, f_n, f_u f_0$, and the cone $\sigma$ giving the above toroidal embedding is just $\R_{\geq 0}f_1+\ldots \R_{\geq 0}f_r+\R_{\geq 0}f_0$. The structure of toroidal embedding on $\xxx_{de}$ is defined similarly with the open set $U_{de}$ given by $S\neq 0$.

We now give some description of the local picture and the relationship between the complex associated to $\xxx_d$ and $\xxx_{de}$.  Consider the stratification on $X$ mentioned before the proof, as in page 195 of \cite{Mum4}. It suffices to look at a stratum $Y_d$ of $\xxx_d$ supported in the closed subscheme defined by $\pi$ because for any strata $Y_d'$ not supported in the $(\pi)$ we may intersect the closure of the strata with the closed subscheme defined by $\pi$ and get a strata $Y_d$ which is a specialisation of $Y_d'$ and we have Star$Y_d'$ naturally injects into Star$Y_d$. We may thus suppose that $\Star Y_d$  locally admits an etale map to a scheme of the form \ref{etx}.  We have a bijective correspondance between the strata on the varieties $\xxx_d$ and $\xxx_{de}$ because the fiber over $R=0$ and $S=0$ are isomorphic under pullback. Denote the stratum corresponding to $Y_d$ by $Y_{de}$. Similar local presentation as \ref{etx} holds for the strata $Y_{de}$ of $\xxx_{de}$.  We have the intrinsically defined $\widetilde{M}^{Y_d}, \widetilde{M}^{Y_{de}}, \widetilde{N}^{Y_d}, \widetilde{N}^{Y_{de}}, \widetilde{\sigma}^{Y_d}, \widetilde{\sigma}^{Y_{de}}$ as in page 196 of \cite{Mum4} (their notation without the tilde). All $\widetilde{\sigma}^{Y_d}$ glued into a conical complex $\widetilde{\Delta}_d$ and all $\widetilde{\sigma}^{Y_{de}}$ glued into a conical complex $\widetilde{\Delta}_{de}$.  Note that we have a globally compatible decomposition $\widetilde{M}^{Y_d}=\Z\oplus M^{Y_d}$, $\widetilde{M}^{Y_{de}}=\Z\oplus M^{Y_{de}}$ with the first factor given by $(\pi)$ and the second factor generated by the irreducible Cartier divisors supported in Star$Y_d-U_d$(Star$Y_{de}-U_{de}$) not supported in $(\pi)$. Similarly we have the decomposition $\widetilde{N}^{Y_d}=\Z\oplus N^{Y_d}$, $\widetilde{N}^{Y_{de}}=\Z\oplus N^{Y_{de}}$, $\widetilde{\sigma}^{Y_d}=\Rp \oplus \sigma^{Y_d}$, $\widetilde{\sigma}^{Y_{de}}=\Rp \oplus \sigma^{Y_{de}}$, $\widetilde{\Delta}_{de}=\Rp\oplus \Delta_{de}$, $\widetilde{\Delta}_{d}=\Rp\oplus\Delta_{d}$. Indeed, if   $Y_d$ is a connected component of $(\pi)\cap \cap_{i\in I}E_i\backslash \cup_{i\notin I}E_i$ with $I$ a subset of the irreducible divisors not supported in $\pi$, and let $Y_{d, \eta}$ be the (unique, by local form) strata of the generic fibre $X_\eta$ in $\cap_{i\in I}E_i\backslash \left(\cup_{i\notin I}E_i\cup (\pi)\right)$ that specializes to $Y_d$, then the $M^{Y_d}, M^{Y_{de}}, N^{Y_d}, N^{Y_{de}}, \sigma^{Y_d}, \sigma^{Y_{de}}, \Delta_d, \Delta_{de}$ are the corresponding lattice, cones and conical complexes associated to the strata $Y_{d, \eta}$ of the (field case) toroidal embedding $U_{d}\inj \xxx_{d, \eta}$ defined as in page 59 of \cite{Mum2}. 

We now use the subdivision given in \cite{Mum3}  to construct a blowup of $\xxx_{de}$. Since $S=R^{1/e}$ is globally defined, it gives a Cartier divisor in $\xxx_{de}$ and hence by the process described in the last paragraph before the proof, it lies in each $M^{Y_{de}}$ compatibly by restriction of Cartier divisors. Hence it defines a function on each cone $\widetilde{\sigma}^{Y_{de}}$ compatibly and hence  a function on $\widetilde{\Delta}_{de}$. Denote this function by $l_{de}$, so it defines a closed subset $\widetilde{\Delta}_{de}^\ast =\{x\in\widetilde{\Delta}^{Y_{de}}| l_{de}(x)=1\}$. Clearly $\widetilde{\Delta}_{de}^\ast=\Rp\oplus\Delta_{de}^\ast$ where $\Delta_{de}^\ast\subset\Delta_{de}$ also defined by $l_{de}=1$ and it is a compact convex polyhedral set because the hyperplane $l_{de}=1$ meets every positive dimensional face of $\Delta_{de}$. $\widetilde{\Delta}_{de}^\ast$ and $\Delta_{de}^\ast$ has integral structure given by various lattices $\widetilde{M}^{Y_d}$ and $M^{Y_d}$.   
The upshot of the discussion from page 105-108 of \cite{Mum2} and  \cite{Mum3} is that we may find an $e$ and a subdivision $\{\sigma_\alpha\}$ of $\Delta_{de}^\ast$ into convex polyhedral sets such that all vertices of all $\sigma_\alpha$ lie in $(\Delta_{de}^\ast)_\Z$, the lattice given by the integral structure mentioned above and that the volume (also given by the integral structure as in page 95 of \cite{Mum2}) of each $\sigma_\alpha$ is $1/(d_\alpha)!$ where $d_\alpha$ is the dimension of $\sigma_\alpha$. Adjoining the origin to the vertices gives a conical decomposition $\{\tau_\alpha\}$ of $\Delta_{de}$ associated to $\{\sigma_\alpha\}$. Now by part (d) of page 197 of \cite{Mum4}, if we let $\widetilde{\Delta}'_{de}$ be the subdivided complex associated to the f.r.p.p decomposition $\{\Rp\oplus\tau_\alpha\}$ of $\widetilde{\Delta}_{de}$, it gives a toroidal embedding $(U_{de}, Z)$ with a map $Z\rightarrow \xxx_{de}$ respecting the inclusion of $U_{de}$. Moreover, the associated map of the conical complex is just $\widetilde{\Delta}'_{de}\rightarrow \widetilde{\Delta}_{de}$, and the integral structure is preserved. 

We may associate an $H_0$ action on $(U_{de}, Z)$ such that the map $Z\rightarrow \xxx_{de}$ is $H_0$-equivariant for the following reason: Part (e) in Page 198 of \cite{Mum4} gives that this map we constructed via a subdivision of the conical complex $\widetilde{\Delta}_{de}'$ is a blowup along an ideal sheaf $\ii$ over $\xxx_{de}$. The ideal $\ii$ over $\Star Y_{de}$ can be written as the sum of $\Of_{\xxx_{de}}(-D)$ for a collection of Cartier divisors $D\in \widetilde{M}^{Y_{de}}$ that satisfy certain condition relevant to the subdivision of the cone $\widetilde{\sigma}^{Y_{de}}$. Now we see from the local form of $\xxx_{de}$ that the $H_0$ action fix every Cartier divisor $D\in \widetilde{M}^{Y_{de}}$ (see Remark \ref{h0act}), and hence preserve the sum of these $\Of_{\xxx_{de}}(-D)$. The asserted $H_0$-equivariance thus follows.

We now use the information on the evaluation of $S$ on $\widetilde{\Delta}'_{de}$ to prove the criterion for regularity and the divisor $S$ vanishing to the order $1$ on any irreducible Weil divisors, as well as giving a local model over $W[S, \upm]'$. By the two condition on each $\sigma_\alpha$, we see that each $\tau_\alpha$ is generated as a cone by a set of vectors $\{f_{\alpha, i}\}_{1\leq i\leq d_\alpha+1}\subset (\sigma_\alpha)_\Z \subset (\tau_\alpha)_\Z$ given by the vertices of $\sigma_\alpha$, and by Lemma 4 on Page 96 of \cite{Mum2}, if we denote the linear subspace generated by $\tau_\alpha$ as $(\tau_\alpha)_\R$, then $\{f_{\alpha, i}\}_{1\leq i\leq d_\alpha+1}$ is a basis for the lattice of $(\tau_\alpha)_\R$ given by the restriction of the integral structure $M^{Y_{de}}$ for some $Y_{de}$ such that $\tau_\alpha\subset (N^{Y_{de}})_\R$. Thus it can be expanded into a basis $\{f_{\alpha, i}\}_{1\leq i\leq n}$ of $N^{Y_{de}}$. Thus by changing the corresponding dual basis as well, we write the torus $T_\eta$ in the local model of the strata corresponding to the cone $\R_{\geq 0}\oplus\tau_\alpha$ as $W[\frac{1}{p}][T_1^{\pm 1}, \ldots, T_n^{\pm 1}]$ with the the $T_i\in X^\ast(T)$ being the dual basis of $f_{\alpha, i}$.  Since the cone $\Rp\oplus\tau_\alpha$ is now generated by $\Rp f_0+\sum_{i=1}^{d_\alpha+1}\Rp f_{\alpha, i}$, the local model is by definition \ref{deftor} isomorphic to $W[T_1, \ldots, T_{d_\alpha+1}, T_{d_\alpha+2}^{\pm 1}, \ldots, T_n^{\pm 1}]$ and hence smooth over $W$. Now the globally defined $U\in X^\ast(T)$ is a unit in the coordinate ring $W[T_1, \ldots, T_{d_\alpha+1}, T_{d_\alpha+2}^{\pm 1}, \ldots, T_n^{\pm 1}]$, hence $U$ can be written as $T_{d_\alpha+2}^{c_{d_\alpha+2}} \cdots T_n^{c_n}$. We also know that $U$ is non-divisible in the finite free abelian group $X^\ast(T)$ as this only involves the integral structure and hence could be checked over $\xxx_{de}$ where an explicit local model gives the result. It follows that there are no nontrivial common divisor of the $c_{d_\alpha+2}, \ldots, c_n$, and hence we can change coordinate and write the coordinate ring as $W[T_1, \ldots, T_{d_\alpha+1}, T_{d_\alpha+2}'^{\pm 1}, \ldots, T_{n-1}'^{\pm 1}, \upm]$ over $W[\upm]$. Furthermore,  since $S\in \widetilde{M}^{Y_{de}}$ evaluated at each $f_{\alpha, i}$, $1\leq i\leq d_\alpha+1$ is $1$ and evaluated at the basis of the first copy of $\Rp$ is $0$, we see that $S=\prod_{i=1}^{d_\alpha+1}T_i\cdot\prod_{j=d_\alpha+2}^{n-1}T_j'^{d_j}\cdot U^s$, and hence over $W[S, \upm]$, we may change coordinate by $T_1'\mapsto T_1\prod_{j=d_\alpha+2}^{n-1}T_j'^{d_j}\cdot U^{s+1}$, $T_i'\mapsto T_i$ for $2\leq i\leq d_\alpha+1$, and $T_j'\mapsto T_j'$ for $d_\alpha+2\leq j\leq n-1$, we have that locally $Z$ admits an etale map to $W[S, \upm, T_1', \ldots, T_{d_\alpha+1}', T_{d_\alpha+2}'^{\pm 1}, \ldots, T_{n-1}'^{\pm 1}]/(SU-\prod_{i=1}^{d_\alpha+1}T_i')$ over $W[S, \upm]$.

We deduce from the local description above that the base change of $Z$ over $W[S, \upm]'$ along $W[S, \upm]'\rightarrow \Of_{F'}: \ S\mapsto \pi^{1/e}, U\mapsto u$, which we denote by $\Y$, admits an etale map to $\Of_{F'}[T'_1, \ldots, T'_{d_\alpha+1}, T_{d_\alpha+2}'^{\pm 1}, \ldots, T_{n-1}'^{\pm 1}]/(\pi^{1/e}u-\prod_{i=1}^{d_\alpha+1}T_i)$. Since $\pi^{1/e}$ is a uniformizer of $\Of_{F'}$, we see that $\Y$ is regular and semistable over $\Of_{F'}$. We may further identify its generic fibre with the projective variety over $F'$ defined by the equation 
\[
u\pi^d(X_1^N+\cdots+X_N^N)=NX_1\cdots X_N
\]
compatible with the $H_0$ action because the generic fibre of $\Y$ is also the base change of $Z_\eta$ along $W[\frac{1}{p}][S, \upm]\rightarrow F': \ S\mapsto \pi^{1/e}, U\mapsto u$. 
The map $Z_\eta\rightarrow \xxx_{de, \eta}$ is an isomorphism outside $S=0$, and hence the above base change of $Z$ is isomorphic to the fibre over $S=\pi^{1/e}$ of $\xxx_{de, \eta}$, which has the above form from its definition and corresponding property of $\xxx_{d, \eta}$. The $H_0$ compatibility follows from the $H_0$ compatibility of the map $Z\rightarrow \xxx_{de}$.

\end{proof}

We now analyze the special fibre of $\Y$ because we want to apply Hyodo-Kato's semistable comparison theorem.

\begin{remark}
\label{link}
Denote the canonical log structure on $\Y$ given by the Cartier divisor $(\pi^{1/e})$ by $\nn$, the special fibre with pullback log structure by $(\overline{\Y}, \overline{\nn})$, the log structure on $Z$ given by the divisor $(S)$ by $M$. Picking a $u'\in W^\times$ that have the same reduction as $u$ in $k^\times$, we denote the base change of $Z$ with pullback log structure from $M$ along  $W[S, U^{\pm 1}]'\rightarrow W[T]':\ U\rightarrow u',\ S\mapsto T$ by $(Y, N)$ (here $\spec W[T]'$ is the open subscheme of $\spec W[T]$ defined by $(T^{de}u')^N\neq 1$), and further base change along $W[T]\rightarrow k$, $T\mapsto 0$ with pullback log structure by $(\overline{Y_1}, \overline{N_1})$. Then we have that as log schemes over $k$ with $H_0$ action, $(\overline{\Y}, \overline{\nn})\cong (\overline{Y_1}, \overline{N_1})$ (because they are both base change of  $(Z, M)$ by $W[S, \upm]'\rightarrow k:$ $S\mapsto 0, U\mapsto \overline{u}$) and $Y$ is a blowup of the projective variety $u'T^{de}(X_1^N+X_2^N+\cdots+X_N^N)=NX_1X_2\cdots X_N$ over $\spec W[T]'$ with equivariant $H_0$-action.

\end{remark}

\section{Comparison Theorem of Log Geometry}

We start by introducing some notation we will use in log geometry, and then we will specialize to the case we work in. We will denote a general log scheme by $(X, M)$ where $X$ is a usual scheme and $M$ denotes the log structure. When $M$ is a monoid, it is to be interpretted as a chart for this log structure. In particular, let $(X, \N, 1\mapsto f)$ be the log scheme determined by the chart $\N$ with $1$ sent to the $f\in \Of_X(X)$ under the log structure map. And let $(X, (0))$ denote the log scheme $X$ with trivial log structure.  When there is only one log structure mentioned on a scheme $X$, we will simply use $X$ to refer to the log scheme $(X, M)$.

We also record the following lemma from \cite{ogus} Chapter 4 Proposition 1.3.1:

\begin{lemma}
\label{pback}
Let 
\[
\begin{tikzcd}
(X', M')\arrow[r, "g"]\arrow[d, "f'"]&
(X, M)\arrow[d, "f"]\\
(Y', N')\arrow[r, "h"]&
(Y, N)
\end{tikzcd}
\]
be a Cartesian diagram of log schemes. Then we have a natural isomorphism 
\[
g^\ast \omega^1_{(X, M)/(Y, N)}\cong \omega^1_{(X', M')/(Y', N')}
\]
\end{lemma}

\

We will work with $(Y, N)$ being a log smooth scheme over $(W[T]', \N, 1\mapsto T)$, where $\spec W[T]'$ is a fixed choice of affine open subscheme of $\spec W[T]$ such that the closed subscheme defined by $T=0$ is contained in $\spec W[T]'$. In particular, we will apply the theory to the $(Y, N)$ we get from last section in Remark \ref{link} over the particular base $(W[T]', \N, 1\mapsto T)$(although same notation).  Denote by $(Y_n, N_n)$ (resp. $(\overline{Y_n}, \overline{N_n})$) the base change of $(Y, N)$ along $(W_n[T]', \N, 1\mapsto T)\rightarrow (W[T]', \N, 1\mapsto T)$(resp. $(W_n, \N, 1\mapsto 0)\rightarrow (W[T]', \N, 1\mapsto T)$). Let $(\overline{Y}, \overline{N})$ be the base change of $(Y, N)$ along $(W, \N, 1\mapsto 0)\rightarrow (W[T]', \N, 1\mapsto T)$.  Note that the closed immersions in the corresponding fibre diagram are all exact and the projection maps to the base are all log smooth.

For a base 4-ple $(S, L, I, \gamma)$, where $S$ is a scheme with $\Of_S$ killed by some integer $p^n$, $L$ is a fine log structure on it (see \cite{HK} (2.6)), $I$ is a quasi-coherent ideal on $S$ and $\gamma$ is a PD structure on $I$, and a fine log scheme $(X, M)$ over $(S, L)$ such that $\gamma$ extends to $X$, we may define the cristalline site and cristalline cohomology as in \cite{HK} (2.15).

Take $(X, M)=(\overline{Y_1}, \overline{N_1})$ and $(S, L)=(W_n, \N, 1\mapsto 0)$ with the ideal $I=(p)$ and its usual PD structure, \cite{HK} define the $i$-th log cristalline cohomology of $(\overline{Y_1}, \overline{N_1})$
as the limit
\[
\varprojlim_{n}H^{i}(((\overline{Y_1}, \overline{N_1})/(W_n,\N, 1\mapsto 0))_{crys},\Of_{\overline{Y_1}/W_n})[\frac{1}{p}]
\]
and they also equip it with a nilpotent operator $N$ that will be the same as the $N$ given by $p$-adic Hodge theory when identified by the comparison isomorphism.

\begin{theorem}
\label{log1}
The above limit is isomorphic to $\Hh^i(\dR_{\overline{Y}/(W, \N, 1\mapsto 0)})[\frac{1}{p}]$  while the operator $N$ is given as the degree $i$ boundary homomorphism of the long exact sequence given by the following exact triangle
    
\begin{tikzcd}
\dR_{\overline{Y}/(W, \N, 1\mapsto 0)}[-1]\arrow[r, "\cdot dlog 1"]
&\dR_{\overline{Y}/(W, (0))}\arrow[r]
&\dR_{\overline{Y}/(W, \N, 1\mapsto 0)}\arrow[r]  &{}
\end{tikzcd}

which is defined at each degree $i$ by taking the $i$-th wedge power  of the following locally split exact sequence of locally free sheaf of modules
\[
\begin{tikzcd}
0\arrow[r]&
\Of_{\overline{Y}}\arrow[r, "\cdot d\log 1"]&
\omega^1_{\overline{Y}/(W, (0))}\arrow[r]&
\omega^1_{\overline{Y}/(W, \N, 1\mapsto 0)}\arrow[r]&
0
\end{tikzcd}
\]
by Theorem 3.2.3 of \cite{ogus} since we note that $(\overline{Y},\overline{N})$ is log smooth over $(W, \N, 1\mapsto 0)$ and we identify $\omega^1_{(W, \N, 1\mapsto 0)/(W, (0))}\cong W$ by $d\log 1\mapsto 1$.
\end{theorem}

\begin{proof}
We will prove the theorem in 3 steps. 

(i) We prove that there exists an $N$-equivariant isomorphism
\begin{equation}
\label{1claim}
H^{i}(((\overline{Y_1}, \overline{N_1})/(W_n,\N, 1\mapsto 0))_{crys},\Of_{\overline{Y_1}/W_n})\cong 
\Hh^{i}(\dR_{\overline{Y_n}/(W_n, \N, 1\mapsto 0)})
\end{equation}
with the operator $N$ on the right hand side given by the degree $i$ boundary homomorphism of the following exact triangle
\[
\begin{tikzcd}
\dR_{\overline{Y_n}/(W_n, \N, 1\mapsto 0)}[-1]\arrow[r, "\cdot dlog 1"]
&\dR_{\overline{Y_n}/(W_n, (0))}\arrow[r]
&\dR_{\overline{Y_n}/(W_n, \N, 1\mapsto 0)}\arrow[r]
&{}
\end{tikzcd}
\]
obtained via various wedge power of the following exact sequence as the process above
\[
\begin{tikzcd}
0\arrow[r]&
\Of_{\overline{Y_n}}\arrow[r, "\cdot d\log 1"]&
\omega^1_{\overline{Y_n}/(W_n, (0))}\arrow[r]&
\omega^1_{\overline{Y_n}/(W_n, \N, 1\mapsto 0)}\arrow[r]&
0
\end{tikzcd}
\]
The transition maps in the limit 
\[
\varprojlim_{n}H^{i}(((\overline{Y_1}, \overline{N_1})/(W_n,\N, 1\mapsto 0))_{crys},\Of_{\overline{Y_1}/W_n})
\]
is compatible with the pullback homomorphism of cohomology \linebreak $\Hh^{i}(\dR_{\overline{Y_{n+1}}/(W_{n+1}, \N, 1\mapsto 0)})\rightarrow \Hh^{i}(\dR_{\overline{Y_n}/(W_n, \N, 1\mapsto 0)})$ under the isomorphism \ref{1claim}.

(ii) We will prove there exists commuting short exact sequences of $W$ modules
\begin{equation}
\begin{tikzcd}[column sep=0.5em]
0\arrow[r]&
\Hh^{i}(\dR_{\overline{Y}/(W, \N, 1\mapsto 0)})/p^{n+1}\arrow[r]\arrow[d]& \Hh^{i}(\dR_{\overline{Y_{n+1}}/(W_{n+1}, \N, 1\mapsto 0)})\arrow[r]\arrow[d]&
\Hh^{i+1}(\dR_{\overline{Y}/(W, \N, 1\mapsto 0)})[p^{n+1}]\arrow[r]\arrow[d, "\cdot p"]&
0\\
0\arrow[r]&
\Hh^{i}(\dR_{\overline{Y}/(W, \N, 1\mapsto 0)})/p^{n}\arrow[r]& \Hh^{i}(\dR_{\overline{Y_{n}}/(W_n, \N, 1\mapsto 0)})\arrow[r]&
\Hh^{i+1}(\dR_{\overline{Y}/(W, \N, 1\mapsto 0)})[p^{n}]\arrow[r]&
0
\end{tikzcd}
\end{equation}
where the middle vertical map is the pullback map, the left vertical map is natural reduction and the right vertical map is multiplication by $p$. All maps in the left commutative square are $N$ equivariant.

(iii) Conclude Theorem \ref{log1}.

\ 

We first prove (i).

Again we recall the base 4-ple $(S, L, I, \gamma)$ being a fine $p^n$ torsion log scheme with a quasi-coherent ideal $I$ and a PD structure on it. In the discussion below, we will always take $I=(p)$ and the usual PD structure on it. This PD structure extends to all schemes $X$ over $S$. Recall the following theory of log cristalline cohomology of crystals. This is the basic case in (2.18) of \cite{HK} combined with (2.17). We will not recall the notions in the proposition but refer to the same reference.

\begin{proposition}
\label{3.a}
For $(X,M)/(S,L)$ log schemes, if there exists a log smooth $(Z,N)/(S,L)$ with a closed immersion $i: (X,M)\hookrightarrow (Z,N)$. Denote the (log) PD envelope of $i: (X,M)\hookrightarrow (Z,N)$ by $(D,M_D)$, then there exists a complex of $\Of_Z$-module $C_{X, Z/S}$ of the form:

\begin{tikzcd}
\Of_D\arrow[r, "\nabla"] & \Of_D\otimes_{\Of_Z}\omega^1_{Z/S}\arrow[r, "\nabla"]&
\Of_D\otimes_{\Of_Z}\omega^2_{Z/S}\arrow[r, "\nabla"]&
\cdots
\end{tikzcd}

where $\omega^1_{Z/S}$ is the log differential module, and Leibniz rules are satisfied:
\[
\nabla(m\otimes \omega)=\nabla(m)\otimes \omega+ m\otimes d\omega \ \ (m\in \Of_D\otimes_{\Of_Z}\omega^i_{Z/S}, \ \ \omega\in \omega^j_{Z/S})
\]
such that the cohomology of this complex as an object in $D(Sh(X_{et}))$ computes the log cristalline cohomology $H^\cdot(((X,M)/(S,L))_{crys}, \Of_{X/S})$.

\end{proposition}

\begin{remark}
\cite{HK} (2.18) also gives that the association of the complex $C_{X,Z/S}$ is natural with respect to the system $(X, M)\rightarrow (Z, N)\rightarrow (S, L, I, \gamma)$.
\end{remark}

Apply this proposition to $(X, M)=(\overline{Y_1}, \overline{N_1})$, $(S, L)=(W_n, \N, 1\mapsto 0)$ and the closed immersion $i: (\overline{Y_1}, \overline{N_1})\rightarrow (\overline{Y_n}, \overline{N_n})$ whose ideal of definition is generated by $p$, hence $(D, M_D)=(\overline{Y_n}, \overline{N_n})$ and by Leibniz rule we see that $C_{\overline{Y_1}, \overline{Y_n}/(W_n, \N, 1\mapsto 0)}=\dR_{\overline{Y_n}/(W_n, \N, 1\mapsto 0)}$ and hence the abstract isomorphism in \ref{1claim} follows. We still need to show it's $N$-equivariant.

There is also a description of the operator $N$ on the log cristalline cohomology in terms of the complexes mentioned above, see \cite{HK} (3.6). We simplify it a little bit.

\begin{proposition}\label{3.11}
Suppose there exist a (log) closed immersion $(X,M)\hookrightarrow (Z,N)$ over $(W_n[T]', \N, 1\mapsto T)$ with $(Z,N)$ log smooth over $(W_n[T]', \N, 1\mapsto T)$. Then there exist an exact triangle
\[
C_{X, Z\langle\rangle/(W_n\langle T\rangle', \N, 1\mapsto T)}[-1]\rightarrow C_{X, Z/(W_n, (0))}\rightarrow C_{X, Z\langle\rangle/(W_n\langle T\rangle', \N, 1\mapsto T)}\rightarrow
\]
where $W_n\langle T\rangle'$ denotes the usual PD envelope of $W_n[T]'$ along the ideal $(t)$ and $Z\langle\rangle$ is the base change of $Z$ along $W_n[T]'\rightarrow W_n\langle T\rangle'$ with the pullback log structure. We refer the reader to \cite{HK} (3.6) for the definition of this exact triangle.

Tensoring the above exact triangle by $\otimes^\Ld_{W_n\langle T\rangle'}W_n$, we get:
\[
C_{X, \overline{Z}/(W_n, \N, 1\mapsto 0)}[-1]\rightarrow C_{X/(W_n, (0))}\otimes^\Ld_{W_n\langle T\rangle'}W_n\rightarrow C_{X, \overline{Z}/(W_n, \N, 1\mapsto 0)}\rightarrow
\]
where $\overline{Z}$ is the base change of $Z$ along $W_n[T]'\rightarrow W_n, t\mapsto 0$ with the pullback log structure, so that the operator $N$ on $H^i(((X,M)/(W_n,\N))_{crys},\Of_{X/W_n})$ is given by the connecting homomorphism of degree $i$ of this exact triangle under the identification given after Proposition \ref{3.a}.

\end{proposition}

In fact, we will use this proposition in our case for $(X, M)=(\overline{Y_1}, \overline{N_1})$ and $(Z, N)=(Y_n, N_n)$. The first is the base change of the second along $(W_n, \N, 1\mapsto 0)\rightarrow (W_n[T]', \N, 1\mapsto T)$. Thus $Z\langle\rangle=(Y_n\langle\rangle, N_n\langle\rangle)$ as log schemes, the latter of which we define to be the base change of $(Y_n,N_n)\rightarrow (W_n[T]', \N, 1\mapsto T)$ along $(W_n\langle T\rangle', \N, 1\mapsto T)$. Also $\overline{Z}=(\overline{Y_n}, \overline{N_n})$ as log schemes.  We would identify the objects appearing in the above proposition in more concrete terms.

Here is a diagram illustrating the situation:

\begin{tikzcd}
(\overline{Y_1}, \overline{N_1})\arrow[dr]\arrow[r]&(Y_n\langle\rangle, N_n\langle\rangle)\arrow[d]\arrow[r]&(Y_n, N_n)\arrow[d]\\
&(W_n\langle T\rangle', \N, 1\mapsto T)\arrow[dr]\arrow[r]&(W_n[T]', \N, 1\mapsto T)\arrow[d]\\
&&(W_n, (0))
\end{tikzcd}

Apply Proposition \ref{3.a} to $(X, M)=(\overline{Y_1}, \overline{N_1})$, $(S, L)=(W_n\langle T\rangle', \N, 1\mapsto T)$ and the closed immersion $(\overline{Y_1}, \overline{N_1})\rightarrow  (Y_n\langle\rangle, N_n\langle\rangle)$. Since the closed immersion is exact and the ideal of definition is generated by $p$ and all the divided powers of $t$, $(D, M_D)=(Y_n\langle\rangle, N_n\langle\rangle)$ in this case. Proposition \ref{3.a} gives that $C_{\overline{Y_1}, Y_n\langle\rangle/(W_n\langle T\rangle', \N, 1\mapsto T)}\cong \dR_{Y_n\langle\rangle/(W_n\langle T\rangle', \N, 1\mapsto T)}$

Apply Proposition \ref{3.a} to $(X, M)=(\overline{Y_1}, \overline{N_1})$, $(S, L)=(W_n, (0))$ and the closed immersion $(\overline{Y_1}, \overline{N_1})\rightarrow  (Y_n, N_n)$. $(Y_n, N_n)$ is log smooth over $(W_n, (0))$ because it is log smooth over $(W_n[T]', \N, 1\mapsto T)$ and $(W_n[T]', \N, 1\mapsto T)$ is log smooth over $(W_n, (0))$ by the smoothness criterion (2.9) in \cite{HK}. In this case $(D, M_D)=(Y_n\langle\rangle, N_n\langle\rangle)$ since the ideal of deinition given by the closed immersion $i$ is generated by $p,t$. Proposition \ref{3.a} shows that $C_{\overline{Y_1}, Y_n/(W_n, (0))}$ is given by a complex 

\begin{tikzcd}
\Of_{Y_n\langle\rangle}\arrow[r, "\nabla"] & \Of_{Y_n\langle\rangle}\otimes_{\Of_{Y_n}}\omega^1_{Y_n/(W_n, (0))}\arrow[r, "\nabla"]&
\Of_{Y_n\langle\rangle}\otimes_{\Of_{Y_n}}\omega^2_{Y_n/(W_n, (0))}\arrow[r, "\nabla"]&\cdots
\end{tikzcd}

Via the application of Proposition \ref{3.a} indicated above, and identifying \linebreak $\omega^i_{Y_n\langle\rangle/(W_n\langle T\rangle', \N, 1\mapsto T)}\cong \Of_{Y_n\langle\rangle}\otimes_{\Of_{Y_n}}\omega^i_{Y_n/(W_n[T]', \N, 1\mapsto T)}\cong W_n\langle T\rangle'\otimes_{W_n[T]'}\omega^i_{Y_n/(W_n[T]', \N, 1\mapsto T)}$,   one can verify that using the recipe given in (3.6) of \cite{HK},   the exact triangle sequence
\[
C_{\overline{Y_1}, Y_n\langle\rangle/(W_n\langle T\rangle', \N, 1\mapsto T)}[-1]\rightarrow C_{\overline{Y_1}, Y_n/(W_n, (0))}\rightarrow C_{\overline{Y_1}, Y_n\langle\rangle/(W_n\langle T\rangle', \N, 1\mapsto T)}\rightarrow
\]

is given by actual chain maps 
\begin{equation}
\label{logex}
\begin{tikzcd}[column sep=0.5em]
0\arrow[r]&\Of_{Y_n\langle\rangle}\otimes_{\Of_{Y_n}} \omega^{i-1}_{Y_n/(W_n[T]', \N, 1\mapsto T)}\arrow[r, "\cdot d\log T"]&
\Of_{Y_n\langle\rangle}\otimes_{\Of_{Y_n}}\omega^i_{Y_n/(W_n, (0))}\arrow[r]&
\Of_{Y_n\langle\rangle}\otimes_{\Of_{Y_n}} \omega^{i}_{Y_n/(W_n[T]', \N, 1\mapsto T)}\arrow[r]&
0
\end{tikzcd}
\end{equation}
which is obtained by tensoring $\Of_{Y_n\langle\rangle}$ with the $i$-th wedge power of the locally split short exact sequence \[
\begin{tikzcd}
0\arrow[r]&
\Of_{Y_n}\arrow[r, "\cdot d\log T"]&
\omega^1_{Y_n/(W_n, (0))}\arrow[r]&
\omega^1_{Y_n/(W_n[T]', \N, 1\mapsto T)}\arrow[r]&
0
\end{tikzcd}
\]

\

Now we claim that the following diagram commutes and the three vertical arrow are isomorphisms :

\begin{tikzcd}[column sep=0.5em]
\label{dr}
\dR_{Y_n\langle\rangle/(W_n\langle T\rangle', \N, 1\mapsto T)}\otimes^\Ld_{W_n\langle T\rangle'}W_n[-1]\arrow[r, "\cdot dlog T"]\arrow[d]
&C_{\overline{Y_1}, Y_n/(W_n, (0))}\otimes^\Ld_{W_n\langle T\rangle'}W_n\arrow[r]\arrow[d]
&\dR_{Y_n\langle\rangle/(W_n\langle T\rangle', \N, 1\mapsto T)}\otimes^\Ld_{W_n\langle T\rangle'}W_n\arrow[d]\\
\dR_{\overline{Y_n}/(W_n, \N, 1\mapsto 0)}[-1]\arrow[r, "\cdot dlog 1"]
&\dR_{\overline{Y_n}/(W_n, (0))}\arrow[r]
&\dR_{\overline{Y_n}/(W_n, \N, 1\mapsto 0)}
\end{tikzcd}

The commutativity is straightforward. The isomorphisms of left and right column just follows form the locally freeness of each modules of differentials and the pullback result Lemma \ref{pback} for modules of differentials.

For the middle vertical arrow, we know that $(\Of_{Y_n\langle\rangle}\otimes_{\Of_{Y_n}}\omega^i_{Y_n/(W_n, (0))})\otimes_{W_n\langle T\rangle'}W_n\cong \omega^i_{\overline{Y_n}/(W_n, (0))}$ and since $\omega^i_{Y_n/(W_n, (0))}\rightarrow \omega^i_{\overline{Y_n}/(W_n, (0))}$ is already surjective and we know the $\nabla$ in $C_{\overline{Y_1}, Y_n/(W_n, (0))}$ satisfy Leibniz rules and hence the differential applied to the image of  $\omega^i_{Y_n/(W_n, (0))} $ in $\Of_{Y_n\langle\rangle}\otimes_{\Of_{Y_n}}\omega^i_{Y_n/(W_n, (0))}$ is just log differential, we see that the differentials in the complex after tensoring are exactly the log de-Rham differentials.                  

Thus, by the last part of Proposition \ref{3.11}, we see that the operator $N$ on \linebreak $H^{i}_{crys}((\overline{Y_1}, \overline{N_1})/(W_n, \N), \Of_{\overline{Y_1}/W_n})$ is given by the degree $i$ connecting homomorphism of the long exact sequence associated to the lower row of the above diagram. Thus (i) is proved.

\

Now we prove (ii).

By proper base change theorem in the context of derived category of coherent sheaves(Section 7.7 of \cite{EGA}) applied to the Cartesian square 
\[
\begin{tikzcd}
\overline{Y_n}\arrow[r, "i_n"]\arrow[d, "\pi_n"]&\overline{Y}\arrow[d, "\pi_\infty"]\\
W_n\arrow[r]&W
\end{tikzcd}
\]
and the exact triangle 
\begin{equation}
\label{lastpage}
\begin{tikzcd}
\dR_{\overline{Y}/(W, \N, 1\mapsto 0)}[-1]\arrow[r, "\cdot dlog 1"]
&\dR_{\overline{Y}/(W, (0))}\arrow[r]
&\dR_{\overline{Y}/(W, \N, 1\mapsto 0)}\arrow[r]
&{}
\end{tikzcd}
\end{equation}

,we see that 
\[
\begin{tikzcd}[column sep=0.5em]
R\pi_{\infty, \ast}\dR_{\overline{Y}/(W, \N, 1\mapsto 0)}\otimes^{\Ld}_{W}W_n[-1]\arrow[r]\arrow[d, "\cong"]&
R\pi_{\infty, \ast}\dR_{\overline{Y}/(W, (0))}\otimes^{\Ld}_{W}W_n\arrow[r]\arrow[d, "\cong"]&
R\pi_{\infty, \ast}\dR_{\overline{Y}/(W, \N, 1\mapsto 0)}\otimes^{\Ld}_{W}W_n\arrow[d, "\cong"]\arrow[r]&{}\\
R\pi_{n, \ast}\dR_{\overline{Y_n}/(W_n, \N, 1\mapsto 0)}[-1]\arrow[r]& 
R\pi_{n, \ast}\dR_{\overline{Y_n}/(W_n, (0))}\arrow[r]& 
R\pi_{n, \ast}\dR_{\overline{Y_n}/(W_n, \N, 1\mapsto 0)}\arrow[r]&{}
\end{tikzcd}
\]
(Note that each term of the de Rham complexes in \ref{lastpage} (as actual chain complexes) is locally free by the log smoothness and that $L i_n^\ast \dR_{\overline{Y}/(W, \N, 1\mapsto 0)}=\dR_{\overline{Y_n}/(W_n, \N, 1\mapsto 0)}$, $L i_n^\ast\dR_{\overline{Y}/(W, (0))}=\dR_{\overline{Y_n}/(W_n, (0))}$ by applying usual pullback at each degree and \ref{pback}).

Applying $R\pi_{\infty, \ast}\dR_{\overline{Y}/(W, \N, 1\mapsto 0)}\otimes^{\Ld}_{W}$ to the exact triangle 
\[
\begin{tikzcd}
W\arrow[r, "\cdot p^n"]&
W\arrow[r]&
W_n\arrow[r]&
{}
\end{tikzcd}
\]
and then taking the long exact sequence, we get a short exact sequence 
\begin{equation}
\label{drseq}
\begin{tikzcd}[column sep=1em]
0\arrow[r]&
\Hh^{i}(\dR_{\overline{Y}/(W, \N, 1\mapsto 0)})/p^n\arrow[r]& \Hh^{i}(\dR_{\overline{Y_n}/(W_n, \N, 1\mapsto 0)})\arrow[r]&
\Hh^{i+1}(\dR_{\overline{Y}/(W, \N, 1\mapsto 0)})[p^n]\arrow[r]&
0
\end{tikzcd}
\end{equation}
where the inclusion is $N$-equivariant since it is induced by the above diagram from the connecting homomorphisms of the diagram 
\[
\begin{tikzcd}[column sep=0.5em]
R\pi_{\infty, \ast}\dR_{\overline{Y}/(W, \N, 1\mapsto 0)}\otimes^{\Ld}_{W}W[-1]\arrow[r]\arrow[d]&
R\pi_{\infty, \ast}\dR_{\overline{Y}/(W, (0))}\otimes^{\Ld}_{W}W\arrow[r]\arrow[d]&
R\pi_{\infty, \ast}\dR_{\overline{Y}/(W, \N, 1\mapsto 0)}\otimes^{\Ld}_{W}W\arrow[d]\arrow[r]&{}\\
R\pi_{\infty, \ast}\dR_{\overline{Y}/(W, \N, 1\mapsto 0)}\otimes^{\Ld}_{W}W_n[-1]\arrow[r]&
R\pi_{\infty, \ast}\dR_{\overline{Y}/(W, (0))}\otimes^{\Ld}_{W}W_n\arrow[r]&
R\pi_{\infty, \ast}\dR_{\overline{Y}/(W, \N, 1\mapsto 0)}\otimes^{\Ld}_{W}W_n\arrow[r]&{}
\end{tikzcd}
\]

Applying $R\pi_{\infty, \ast}\dR_{\overline{Y}/(W, \N)}\otimes^{\Ld}_{W}$ to the morphism between exact triangles
\[
\begin{tikzcd}
W\arrow[r, "\cdot p^{n+1}"]\arrow[d, "\cdot p"]&
W\arrow[r]\arrow[d, "="]&
W_{n+1}\arrow[r]\arrow[d]&
{}\\
W\arrow[r, "\cdot p^n"]&
W\arrow[r]&
W_n\arrow[r]&
{}
\end{tikzcd}
\]

and take long exact sequence will give us the desired commutative diagram 

\begin{equation}
\label{conn}
\begin{tikzcd}[column sep=0.5em]
0\arrow[r]&
\Hh^{i}(\dR_{\overline{Y}/(W, \N, 1\mapsto 0)})/p^{n+1}\arrow[r]\arrow[d]& \Hh^{i}(\dR_{\overline{Y_{n+1}}/(W_{n+1}, \N, 1\mapsto 0)})\arrow[r]\arrow[d]&
\Hh^{i+1}(\dR_{\overline{Y}/(W, \N, 1\mapsto 0)})[p^{n+1}]\arrow[r]\arrow[d, "\cdot p"]&
0\\
0\arrow[r]&
\Hh^{i}(\dR_{\overline{Y}/(W, \N, 1\mapsto 0)})/p^{n}\arrow[r]& \Hh^{i}(\dR_{\overline{Y_{n}}/(W_n, \N, 1\mapsto 0)})\arrow[r]&
\Hh^{i+1}(\dR_{\overline{Y}/(W, \N, 1\mapsto 0)})[p^{n}]\arrow[r]&
0
\end{tikzcd}
\end{equation}

in (ii) where the left commutative square is $N$ equivariant. Thus (ii) is proved.

\ 

To prove (iii), take inverse limit over the short exact sequences in \ref{conn} indexed by $n$, we get
\[
\begin{tikzcd}[column sep=0.5em]
0\arrow[r]&
\varprojlim_{n}\Hh^{i}(\dR_{\overline{Y}/(W, \N, 1\mapsto 0)})/p^n\arrow[r]& \varprojlim_{n}\Hh^{i}(\dR_{\overline{Y_n}/(W_n, \N, 1\mapsto 0)})\arrow[r]&
\varprojlim_{\cdot p}\Hh^{i+1}(\dR_{\overline{Y}/(W, \N, 1\mapsto 0)})[p^n]
\end{tikzcd}
\]
where the inclusion is $N$ equivariant.

Note that all $\Hh^{j}(\dR_{\overline{Y}/(W, \N, 1\mapsto 0)})$ are finitely generated over $W$ (by truncation and that proper pushforward preserve coherent sheaves). So \linebreak $\varprojlim_{n}\Hh^{i}(\dR_{\overline{Y}/(W, \N, 1\mapsto 0)})/p^n\cong \Hh^{i}(\dR_{\overline{Y}/(W, \N, 1\mapsto 0)})$ and \linebreak $\varprojlim_{\cdot p}\Hh^{i+1}(\dR_{\overline{Y}/(W, \N, 1\mapsto 0)})[p^n]=0$. Thus (iii) is proved.

\end{proof}

We would like to prove another GAGA type result.

Let $(X, M)$ be a log scheme smooth over $(\C, \N, 1\mapsto 0)$ and proper as a scheme. Let $(X^\an, M^\an)$ be the analytic log scheme associated to it which is also smooth over the analytic point $(\pt, \N, 1\mapsto 0)$.

\begin{proposition}

There exists an $N$ equivariant isomorphism
\[
\Hh^{i}(\dR_{X/(\C, \N, 1\mapsto 0)})\cong \Hh^{i}(\dR_{X^\an/(\pt, \N, 1\mapsto 0)})
\]
where the $N$ on each hypercohomology is defined by degree $i$ boundary morphism of the exact triangles obtained similar to the one in Theorem \ref{criscp}

\begin{tikzcd}
\dR_{X/(\C, \N, 1\mapsto 0)}[-1]\arrow[r, "\cdot dlog 1"]
&\dR_{X/(\C, (0))}\arrow[r]
&\dR_{X/(\C, \N, 1\mapsto 0)}\arrow[r]
&{}
\end{tikzcd}

\begin{tikzcd}
\dR_{X^\an/(\pt, \N, 1\mapsto 0)}[-1]\arrow[r, "\cdot dlog 1"]
&\dR_{X^\an/(\pt,(0))}\arrow[r]
&\dR_{X^\an/(\pt, \N, 1\mapsto 0)}\arrow[r]
&{}
\end{tikzcd}

\end{proposition}

\begin{proof}
We can apply the exact GAGA functor $\mathcal{F}\mapsto \mathcal{F}^{\mathrm{an}}$ from the category of sheaves of $\Of_{X}$ modules to the category of sheaves of $\Of_{X^\an}$ modules. Let $\lambda$ be the morphism $X^\an\rightarrow X$. Thus it suffices to show there is a canonical isomorphism between 

\begin{tikzcd}
\lambda^\ast\dR_{X/(\C, \N, 1\mapsto 0)}[-1]\arrow[r, "\cdot dlog 1"]\arrow[d, "\cong"]
&\lambda^\ast\dR_{X/(\C, (0))}\arrow[r]\arrow[d, "\cong"]
&\lambda^\ast\dR_{X/(\C, \N, 1\mapsto 0)}\arrow[r]\arrow[d, "\cong"]
&{} \\
\dR_{X^\an/(\pt, \N, 1\mapsto 0)}[-1]\arrow[r, "\cdot dlog 1"]
&\dR_{X^\an/(\pt, (0))}\arrow[r]
&\dR_{X^\an/(\pt, \N, 1\mapsto 0)}\arrow[r]
&{}
\end{tikzcd}
    
Commutativity of diagram is clear from the definition of $log 1$ as the image of $1$ of the monoid in the differential modules. The isomorphism of complexes actually holds termwise. Namely, we could prove the stronger statement that for any log scheme $(X, M)/(Z, L)$ over $\C$, if we denote its analytification by $(X^\an, M)/(Z^\an, L)$ then $\lambda^\ast\omega^i_{(X, M)/(Z, L)}\cong \omega^i_{(X^\an, M)/(Z^\an, L)}$ for any non-negative integer $i$. It suffices to prove for $i=1$.  We know from definition that $\omega^1_{(X, M)/(Z, L)}$ is $\Omega^1_{X/Z}\oplus M\otimes\Of_X$ modulo the relation $di(m)\oplus-m\otimes i(m)$ and $0\oplus l\otimes 1$ for all $m\in M$ and $l\in L$ where $i$ is the structure morphism $M\rightarrow \Of_X$. Similarly $\omega^1_{(X^\an, M)/(Z^\an, L)}$ is $\Omega^1_{X^\an/Z^\an}\oplus M\otimes\Of_{X^\an}$ modulo the relation $di(m)\oplus-m\otimes i(m)$ and $0\oplus l\otimes 1$ for all $m\in M$ and $l\in L$ where $i$ is the structure morphism $M\rightarrow \Of_{X^\an}$. We conclude the case $i=1$ by noting the fact that $\lambda^\ast \Omega^1_{X/Z}\cong\Omega^1_{X^\an/Z^\an}$.

\end{proof}

We apply this proposition to $(X, M)=(\overline{Y}_\C, \overline{N}_\C)$ and its analytification $(\overline{Y^\an}, \overline{N^\an})$. Since $Z$ is a blowup of the projective variety $\xxx_{de}$(Under the notation of previous section), as the base change of $Z$ along $W[T]'\rightarrow \C$, $T\mapsto 0$, $\overline{Y}_\C$ is projective as well. Hence the condition of the proposition is satisfied and we may reduce the computation of the operator $N$ to the analytic setting.

\section{Proof of Theorem \ref{mt}}
Recall the hypothesis that $F$ is a characteristic $0$ local field containing $\zeta_N$ whose residue characteristic equals to $p$ and does not divide $N$.

It suffices to prove that $V_{\lambda, t^{-1}}$ is regular ordinary as a $G_{F'}$ representation for a finite extension $F'/F$ because of the following lemma.

\begin{lemma}
If a representation $r: G_F\rightarrow GL_n(\overline{\Q}_p)$ satisfy that its restriction $r|_{G_{F'}}$ is de Rham regular and ordinary for some finite extension $F'/F$, then the same holds for $r$.
\end{lemma}

\begin{proof}
We may assume without loss of generality that $F'/F$ is Galois. By the interpretation of the regular and ordinary condition as in Theorem 6.1.2 of \cite{tap}, we see that 
\[
r|_{G_{F'}}\sim \left(
\begin{matrix}
\psi_{1}&\ast&\ast&\ast\\
0&\psi_{2}&\ast&\ast\\
\vdots&\ddots&\ddots&\ast\\
0&\cdots&0&\psi_{n}
\end{matrix}
\right)
\]
where for each $i=1,\ldots,n$  the character $\psi_{i}$ : $G_{F'}\rightarrow\overline{\mathbb{Q}}_{p}^{\times}$  agrees with the character

$$
\sigma\in I_{F'}\mapsto\prod_{\tau\in \mathrm{H}\mathrm{o}\mathrm{m}(F',\overline{\mathbb{Q}}_{p})}\tau(\mathrm{Art}_{F'}^{-1}(\sigma))^{\mu_{\tau, i}}
$$
on an open subgroup of the inertia group $I_{L}$, for some tuple $\mu_{\tau, i}$ satisfying for each fixed $\tau$, $\mu_{\tau, 1}<\cdots<\mu_{\tau, n}$. It suffices to show the $G_F$(actually $\gal(F'/F)$) action preserve the filtration given by the above triangular form and $\mu_{\tau, i}=\mu_{\tau\sigma, i}$ for any $\tau\in\mathrm{Hom}(F', \overline{\Q}_p)$ and $\sigma\in\gal(F'/F)$.  

Take a $\sigma\in\gal(F'/F)$, consider the one dimensional subspace $\overline{\Q}_pe_1$ underlying $\psi_1$ and the subspace $L_1=\overline{\Q}_p\sigma e_1$. $G_{F'}$ acts on $L_1$ with the $\tau$ HT weight $-\mu_{\tau\sigma, 1}$. Now if there exists some $\tau$ such that $\mu_{\tau, 1}\neq\mu_{\tau\sigma, 1}$, there must exist some $\tau$ such that $\mu_{\tau\sigma, 1}<\mu_{\tau, 1}$, thus the $\tau$ HT weights of $L_1$ is strictly greater than the $\tau$ HT weights of $\psi_i$ for any $i=1,\ldots, n$(Following from the strict ordering of the various $\mu_{\tau, i}$). From this we see $\mathrm{Hom}(L_1. \psi_i)=0$  for any $i=1,\ldots, n$(otherwise their $\tau$ HT weights must be same for all $\tau$) and thus $\mathrm{Hom}(L_1, V)=0$ where $V$ denotes the underlying vector space of $r$, which is a contradiction.

Thus $\mu_{\tau, 1}=\mu_{\tau\sigma, 1}$ for any $\tau\in\mathrm{Hom}(F', \overline{\Q}_p)$ and $\sigma\in\gal(F'/F)$. By similar argument, we still have $\mathrm{Hom}(L_1. \psi_i)=0$  for any $i=2,\ldots, n$, hence $\mathrm{Hom}(L_1, V)=\mathrm{Hom}(L_1, \psi_1)$. The left hand side is nonzero and $L_1$ and the vector space underlying $\psi_1$ are both of $1$ dimensional, thus we must have $\sigma$ preserve the $1$ dimensional vector space underlying $\psi_1$. We proved that the first step of the filtration is preserved under $\gal(F'/F)$ action. Quotienting out the subspace underlying $\psi_1$, we can use the same argument to deduce the rest of the claim.

\end{proof}

Note that the $F'/F$ we will use is totally ramified, hence the residue field is still $k$.

The following $p$-adic Hodge theoretic lemma will be used.

\begin{lemma}
Suppose that $F_v$ is a characteristic $0$ local field over $\Q_p$ and
$$
r: \gal(\overline{F_v}/F_v)\rightarrow GL_{n}(\overline{\mathbb{Q}}_{p})
$$
is semi-stable. Suppose moreover that for every $\tau:F_v^0\inj \overline{\Q}_p$, the operator $N$ on the $n$-dimensional vector space
\[
(r\otimes_{\tau, F_v^0}B_{\st})^{\mathrm{Gal}(\overline{F}_v/F_v)}
\]
is maximally nilpotent, i.e. the smallest $j$ such that $N^j=0$ is $n$. Then $r$ is regular and ordinary.
\end{lemma}

\begin{proof}
This Lemma 2.2 (2) of \cite{BLGHT}.
\end{proof}

Thus it suffices to prove that the operator $N$ on the space $D_{st, \sigma}(V_{\lambda, t^{-1}}):= (V_{\lambda, t^{-1}}\otimes_{\sigma, (F')^0}B_{\st})^{\mathrm{Gal}(\overline{F'}/F')}$ is maxinally nilpotent, where $(F')^0$ is the maximally unramified subextension  of $F'$ and $\sigma: (F')^0\inj \overline{\Q(\zeta_N)}_\lambda$ is any fixed embedding of  $(F')^0$ and $\lambda$ is a place  of $\Q(\zeta_N)$ over $p$. 
Applying the main comparison theorem of \cite{Tsuji} to the semistable model $\Y$ we get from Theorem \ref{ssbp}, we have
\[
D_{st, \sigma}(V_{\lambda, t^{-1}})\cong H^{N-2}_{crys}((\overline{\Y}, \overline{\nn})/(W, \N, 1\mapsto 0))[\frac{1}{p}]^{H_0, \sigma^{-1}\chi}\otimes_{(F')^0, \sigma}\overline{\Q(\zeta_N)}_\lambda
\]
as filtered $(\phi, N)$-modules, where the right hand side is Hyodo-Kato's log cristalline cohomology.

Now apply Proposition \ref{ss} and Remark \ref{link} to get the variety $Y$ with $H_0$ action over $\spec W[T]$ and the remark gives that
\[
H^{N-2}_{crys}((\overline{\Y}, \overline{\nn})/(W, \N, 1\mapsto 0))[\frac{1}{p}]^{H_0, \sigma^{-1}\chi}\otimes_{(F')^0, \sigma}\overline{\Q(\zeta_N)}_\lambda\cong
\]
\[
H^{N-2}_{crys}((\overline{Y_1}, \overline{N_1})/(W, \N, 1\mapsto 0))[\frac{1}{p}]^{H_0, \sigma^{-1}\chi}\otimes_{(F')^0, \sigma}\overline{\Q(\zeta_N)}_\lambda
\]

So we only need to identify the limit 
\[
\varprojlim_{n}H^{N-2}(((\overline{Y_1}, \overline{N_1})/(W_n,\N, 1\mapsto 0))_{crys},\Of_{\overline{Y_1}/W_n})[\frac{1}{p}]
\]
and describe the operator $N$ on its $H_0$ eigenspace determined by the character $\sigma^{-1}\chi$.

Now by Theorem \ref{criscp} applied to the log scheme $(Y, N)$ where $Y$ is as above and $N$ is the log structure on it given by the divisor $T$ over $(W[T], \N, 1\mapsto T)$ and $i=N-2$, we may identify the above limit with the de Rham cohomology of $(\overline{Y}, \overline{N})/(W, \N, 1\mapsto 0)$. And it suffice to prove that the operator $N$ defined as the degree $N-2$ boundary homomorphism of the long exact sequence given by the following exact triangle
    
\begin{tikzcd}
\dR_{\overline{Y}/(W, \N, 1\mapsto 0)}[-1]\arrow[r, "\cdot dlog 1"]
&\dR_{\overline{Y}/(W, (0))}\arrow[r]
&\dR_{\overline{Y}/(W, \N, 1\mapsto 0)}\arrow[r]
&{}
\end{tikzcd}

is maximally nilpotent on the $\sigma^{-1}\chi$ eigenspace of the $H_0$ action.

This algebraic statement can be verified over $\C$. So let $(\overline{Y}_\C, \overline{N}_\C)$ denote the base change of $(\overline{Y}, \overline{N})$ along $(W, \N, 1\mapsto 0)\rightarrow(\C, \N, 1\mapsto 0)$ for some fixed embedding $\tau: W\rightarrow \C$. We are reduced to show that the $N$ on $\Hh^{N-2}(\dR_{\overline{Y}_\C/(\C, \N, 1\mapsto 0)})$ given as the degree $N-2$ boundary map of the long exact sequence given by the exact triangle

\begin{tikzcd}
\dR_{\overline{Y}_\C/(\C, \N, 1\mapsto 0)}[-1]\arrow[r, "\cdot dlog 1"]
&\dR_{\overline{Y}_\C/(\C, (0))}\arrow[r]
&\dR_{\overline{Y}_\C/(\C, \N, 1\mapsto 0)}\arrow[r]
&{}
\end{tikzcd}

is maximally nilpotent on the $\tau\sigma^{-1}\chi$ eigenspace of the $H_0$ action.

Let $(S, \N, 1\mapsto T)$ be the analytic disc centered at $0$ with radius $1$, with coordinate $T$ in $\C$ and $\pi^\an: Y^\an\rightarrow S$ be the analytic space associated to $Y$ pulled back to $S$ from the affine line. We have the log structure $N^\an$ on $Y^\an$ given by $N$ on $Y$. We write $(\overline{Y^\an}, \overline{N^\an})$ as the base change of $(Y^\an, N^\an)$ along $(\pt, \N, 1\mapsto 0)\rightarrow (S, \N, 1\mapsto T)$, which is clearly also the analytification of $(\overline{Y}_\C, \overline{N}_\C)$. Now we apply Theorem \ref{injty} to see there exists an $N$ and $H_0$ equivariant isomorphism
\[
\Hh^{N-2}(\dR_{\overline{Y}_\C/(\C, \N, 1\mapsto 0)})\cong \Hh^{N-2}(\dR_{\overline{Y^\an}/(\pt, \N, 1\mapsto 0)})
\]
We are reduced to prove that the operator $N$ on the $\sigma^{-1}\chi$ eigenspace of \linebreak $\Hh^{N-2}(\dR_{\overline{Y^\an}/(\pt, \N, 1\mapsto 0)})$ is maximally nilpotent.

\

Apply \cite{ill} 2.2.2 to the setting $X=Y^\an$, $S=S$ with their mentioned log structure.

\cite{ill} 2.2.2 says $R^{i}\pi^\an_{\ast}\dR_{Y^\an/(S, \N, 1\mapsto T)}$ is locally free for all $i$(in particular for $i=N-2$) and

\begin{equation}
\label{22}
\left(R^{N-2}\pi^\an_{\ast}\dR_{Y^\an/(S, \N, 1\mapsto T)}\right)\otimes_{\Of_S}\C_{\{0\}}\cong \Hh^{N-2}(\dR_{\overline{Y^\an}/(\pt, \N, 1\mapsto 0)})
\end{equation}

(Here, we identify the $\omega_Y^\cdot $ in \cite{ill} 2.2.2, which is defined there as the derived pullback of $\dR_{Y^\an/(S, \N, 1\mapsto T)}$ to the point $0$,  with $\dR_{\overline{Y^\an}/(\pt, \N, 1\mapsto 0)}$ by Lemma \ref{pback} and that by locally freeness we can pullback termwise for $\dR_{Y^\an/(S, \N, 1\mapsto T)}$.)

We claim this isomorphism \ref{22} is $N$ equivariant, where the $N$ on the left hand side is the reduction of the Gauss-Manin connection 
\[
R^{N-2}\pi^\an_{\ast}\dR_{Y^\an/(S, \N, 1\mapsto T)}\rightarrow R^{N-2}\pi^\an_{\ast}\dR_{Y^\an/(S, \N, 1\mapsto T)}\otimes_{\Of_S}\omega^1_{(S, \N, 1\mapsto T)/\C}
\]
that is (see 2.2.1.2 and 2.2.1.3 of \cite{ill}) the degree $N-2$ connecting homomorphism of the following exact triangle (the definition is completely similar to that in step (2) of section 3) when applying $R\pi^\an_{\ast}$, if we identify $\omega^1_{(S, \N, 1\mapsto T)/\C}\cong \Of_S$ by $dlog T\mapsto 1$:

\begin{tikzcd}
\dR_{Y^\an/(S, \N, 1\mapsto T)}[-1]\arrow[r, "\cdot dlog T"]
&\dR_{Y^\an/(\pt, (0))}\arrow[r]
&\dR_{Y^\an/(S, \N, 1\mapsto T)}\arrow[r]
&{}
\end{tikzcd}   

and the $N$ on the right hand side is defined in step (2) of section 3 and is the operator we reduced to calculate.

Granting the following lemma, we are reduced to calculate the residue of the Gauss-Manin connection at $0$ on the locally free sheaf $R^{N-2}\pi^\an_{\ast}\dR_{Y^\an/(S, \N, 1\mapsto T)}$, and that is linked to monodromy of this locally free sheaf by \cite{ill} (2.2.3). The proof of the lemma is quite formal, the reader is recommended to skip it.

\begin{lemma}
The above isomorphism \ref{22} is $N$-equivariant.
\end{lemma}

\begin{proof}

By proper base change applied to the cartesian diagram
\[
\begin{tikzcd}
\overline{Y^\an}\arrow[r, "i_0"]\arrow[d]&
Y^\an\arrow[d, "\pi^\an"]\\
\pt\arrow[r]&
S
\end{tikzcd}
\]
and the exact triangle of complexes over $Y^\an$

\begin{tikzcd}
\dR_{Y^\an/(S, \N, 1\mapsto T)}[-1]\arrow[r, "\cdot dlog T"]
&\dR_{Y^\an/(\pt, (0))}\arrow[r]
&\dR_{Y^\an/(S, \N, 1\mapsto T)}\arrow[r]
&{}
\end{tikzcd}

we get an isomorphism of exact triangles
\[
\begin{tikzcd}[column sep=0.5em]
\Hh^\cdot(\dR_{\overline{Y^\an}/(\pt, \N, 1\mapsto 0)})[-1]\arrow[r]\arrow[d]&
\Hh^\cdot(\dR_{\overline{Y^\an}/(\pt, (0))})\arrow[r]\arrow[d]&
\Hh^\cdot(\dR_{\overline{Y^\an}/(\pt, \N, 1\mapsto 0)})\arrow[d]\arrow[r]&{}\\
R\pi^\an_\ast\dR_{Y^\an/(S, \N, 1\mapsto T)}\otimes^\Ld_{\Of_S}\C_{\{0\}}[-1]\arrow[r]
&R\pi^\an_\ast\dR_{Y^\an/(\pt, (0))}\otimes^\Ld_{\Of_S}\C_{\{0\}}\arrow[r]
&R\pi^\an_\ast\dR_{Y^\an/(S, \N, 1\mapsto T)}\otimes^\Ld_{\Of_S}\C_{\{0\}}\arrow[r]&{}
\end{tikzcd}
\]

Thus $\Hh^{N-2}(\dR_{\overline{Y^\an}/(\pt, \N, 1\mapsto 0)})\cong \Hh^{N-2}(R\pi^\an_\ast\dR_{Y^\an/(S, \N, 1\mapsto T)}\otimes^\Ld_{\Of_S}\C_{\{0\}})$ is $N$ equivariant where the $N$ on the right hand side is given as the degree $N-2$ boundary momorphism of the lower exact triangle. Now to prove the lemma, it suffices to prove that the surjection $R^{N-2}\pi^\an_{\ast}\dR_{Y^\an/(S, \N, 1\mapsto T)}\twoheadrightarrow \Hh^{N-2}(\dR_{\overline{Y^\an}/(\pt, \N, 1\mapsto 0)})$ is $N$ equivariant, which factor as the composition of the ($N-2$)-th cohomology of the map $R\pi^\an_\ast\dR_{Y^\an/(S, \N, 1\mapsto T)}\rightarrow R\pi^\an_\ast\dR_{Y^\an/(S, \N, 1\mapsto T)}\otimes^\Ld_{\Of_S}\C_{\{0\}}$ and the inverse of the isomorphism $R\pi^\an_\ast\dR_{Y^\an/(S, \N, 1\mapsto T)}\otimes^\Ld_{\Of_S}\C_{\{0\}}\rightarrow \Hh^\cdot(\dR_{\overline{Y^\an}/(\pt, \N, 1\mapsto 0)})$ of the left or right column in the above diagram 

Now it is reduced to show the $N$ equivariance of the natural map $R^{N-2}\pi^\an_\ast\dR_{Y^\an/(S, \N, 1\mapsto T)}\rightarrow \Hh^{N-2}(R\pi^\an_\ast\dR_{Y^\an/(S, \N, 1\mapsto T)}\otimes^\Ld_{\Of_S}\C_{\{0\}})$. But this clearly follows from the commutative diagram
\[
\begin{tikzcd}[column sep=0.3em]
R\pi^\an_\ast\dR_{Y^\an/(S, \N, 1\mapsto T)}[-1]\arrow[r]\arrow[d]
&R\pi^\an_\ast\dR_{Y^\an/(S, (0))}\arrow[r]\arrow[d]
&R\pi^\an_\ast\dR_{Y^\an/(S, \N, 1\mapsto T)}\arrow[d]
\\
R\pi^\an_\ast\dR_{Y^\an/(S, \N, 1\mapsto T)}\otimes^\Ld_{\Of_S}\C_{\{0\}}[-1]\arrow[r]
&R\pi^\an_\ast\dR_{Y^\an/(\pt, (0))}\otimes^\Ld_{\Of_S}\C_{\{0\}}\arrow[r]
&R\pi^\an_\ast\dR_{Y^\an/(S, \N, 1\mapsto T)}\otimes^\Ld_{\Of_S}\C_{\{0\}}
\end{tikzcd}
\]

\end{proof}

Thus it now is reduced to show the residue $N$ at $0$ of the $H_0$-$\tau\sigma^{-1}\chi$ eigenpart of $R^{N-2}\pi^\an_\ast\dR_{Y^\an/(S, \N, 1\mapsto T)}$ is maximally nilpotent.  But by Corollary 2.2.3 of \cite{ill}, if we identify the fibre of the  $H_0$-$\tau\sigma^{-1}\chi$ eigenpart of $R^{N-2}\pi^\an_\ast\dR_{Y^\an/(S, \N, 1\mapsto T)}$ at $0$ with the fibre at $s$, here $s\neq 0$ and $(s^{de}u')^N\neq 1$, then the monodromy $T_s$ at $s$ around $0$ is identified with $\expo(-2\pi i N)$. We see from Proposition \ref{ss} and Remark \ref{link} that outside $T=0$, the family $Y^\an$ is just the projective variety $u'T^{de}(X_1^N+X_2^N+\cdots+X_N^N)=NX_1X_2\cdots X_N$ over the punctured analytic disc with coordinate $T$, where $u'$ should be viewed as its image in $\C$. As $Y^\an$ is (over the punctured unit disc) the base change of the projective variety $u'T(X_1^N+X_2^N+\cdots+X_N^N)=NX_1X_2\cdots X_N$ along $S\rightarrow S$, $T\mapsto T^{de}$, we know the monodromy $T_s$ can be identified with $\rho_s(\gamma_\infty)^{de}$ which is again maximally unipotent from the condition of our Theorem \ref{mt} since the coefficient field $\C$ is of characteristic $0$. Now the maximal nilpotence of $N$ follows from  the monodromy $T_1=\expo(-2\pi iN)$ on the eigenspace of the $H_0$ action being maximally unipotent.


\bibliographystyle{amsalpha}
\bibliography{mybibliography}

\end{document}